\documentclass[10pt, letterpaper]{amsart}
\usepackage[utf8]{inputenc}

\usepackage[table]{xcolor}
\usepackage{amsmath}
\usepackage{amssymb}
\usepackage{amsfonts}
\usepackage{amsthm}
\usepackage[foot]{amsaddr}
\usepackage{pythonhighlight}
\usepackage{tikz}

\usepackage{algorithm}
\usepackage{algpseudocode}
\usepackage{caption}

\usepackage[letterpaper, margin=1in]{geometry}
\usepackage{hyperref}[colorlinks=true]
\usepackage[normalem]{ulem}

\newtheorem{proposition}{Proposition}

\newtheorem{remark}{Remark}

\newcommand{\R}{\mathbb R}
\DeclareMathOperator*{\argmax}{\arg\max}
\DeclareMathOperator*{\argmin}{\arg\min}

\newcommand{\reviewerA}[1]{{#1}}
\newcommand{\reviewerB}[1]{{#1}}
\newcommand{\reviewerCommon}[1]{{#1}}

\usepackage{color,soul}
\title{Neural Empirical Interpolation Method for Nonlinear Model Reduction}

\author[M. Hirsch]{Max Hirsch}
\address[MH]{Department of Mathematics, University of California Berkeley, Berkeley, CA 94720, USA.}
\email[]{mhirsch@berkeley.edu}
\author[F. Pichi]{Federico Pichi}
\address[FP]{mathLab, Mathematics Area, SISSA, via Bonomea 265, I-34136 Trieste, Italy.}
\email[]{fpichi@sissa.it}
\author[J. S. Hesthaven]{Jan S. Hesthaven}
\address[JSH, FP]{MCSS, École Polytechnique Fédérale de Lausanne, 1015 Lausanne, Switzerland.}
\email[]{jan.hesthaven@epfl.ch}

\begin{document}

\begin{abstract}
  In this paper, we introduce the neural empirical interpolation method (NEIM), a neural network-based alternative to the discrete empirical interpolation method for reducing the time complexity of computing the nonlinear term in a reduced order model (ROM) for a parameterized nonlinear partial differential equation. NEIM is a greedy algorithm which accomplishes this reduction by approximating an affine decomposition of the nonlinear term of the ROM, where the vector terms of the expansion are given by neural networks depending on the ROM solution, and the coefficients are given by an interpolation of some ``optimal'' coefficients. Because NEIM is based on a greedy strategy, we are able to provide a basic error analysis to investigate its performance. NEIM has the advantages of being easy to implement in models with automatic differentiation, of being a nonlinear projection of the ROM nonlinearity, of being efficient for both nonlocal and local nonlinearities, and of relying solely on data and not the explicit form of the ROM nonlinearity. We demonstrate the effectiveness of the methodology on solution-dependent and solution-independent nonlinearities, a nonlinear elliptic problem, and a nonlinear parabolic model of liquid crystals.
\end{abstract}

\maketitle

\begin{center}
    \textbf{Code availability:} \url{https://github.com/maxhirsch/NEIM}
\end{center}

\section{Introduction}
\label{sec:introduction}

In the numerical solution of Partial Differential Equations (PDEs) \cite{QuarteroniNumericalApproximationPartial1994}, it is often necessary to study the behavior of the physics-based model corresponding to many different parameter configurations. This is the case, for example, in inverse problems, design, and uncertainty quantification, where physical or geometric properties of the problem need to be investigated. Approximating the solutions of the PDE via high-fidelity simulations usually involves solving a system of equations of large dimensionality, and obtaining many-query and real-time simulations becomes prohibitively expensive. Reduced order modeling (ROM) \cite{BennerSnapshotBasedMethodsAlgorithms2020, BennerSystemDataDrivenMethods2021} consists of a wide class of approaches in which such solutions are approximated via surrogate models with lower dimensional systems of equations. Because of this reduced or latent representation, the system can be solved more efficiently, allowing for a reduction of the computational cost.
ROM techniques are divided into an offline stage and an online stage, where the information collected via high fidelity solvers during the former is exploited for fast online evaluations. This is possible thanks to the \textit{reduced basis} space, a low-dimensional space spanned by an orthonormal basis which can be obtained via well-known approaches such as the Greedy method or the Proper Orthogonal Decomposition (POD) \cite{hesthavenCertifiedReducedBasis2015,QuarteroniReducedBasisMethods2016}. Within the POD-based model order reduction framework, the high fidelity or full-order approximation of the solution is computed and stored offline for several parameter instances, forming the so-called \textit{snapshots}. Then, a singular value decomposition of the snapshots matrix is performed to retain the hopefully few left singular vectors capturing most of the system's behavior. A Galerkin projection is then performed on the subspace encoded by the extracted reduced basis, and the high fidelity problem is reformulated in a lower dimensional space. In the online stage, the lower dimensional problem can be solved efficiently and accurately for many different parameter configurations.

To obtain such computational speedup, two main assumptions need to be satisfied: (i) the problem is reducible, meaning that we can compress the information with a small number of basis functions, i.e., its Kolmogorov n-width has a fast decay\reviewerCommon{,} and (ii) we can efficiently assemble the parameter-dependent quantities thanks to the affine-decomposition property, allowing us to have an online stage that does not depend on the potentially high number of degrees of freedom of the problem.  The former is related to the physical complexity of the model, while the latter can be satisfied choosing ad-hoc strategies to restore the computational savings.

Indeed, when dealing with linear PDEs with affine parameter dependence, the POD-Galerkin approach results in a linear system of equations of dimension much smaller than the original problem, leading to consistent speedups. However, when the PDE is nonlinear, POD-Galerkin will not in general result in significant speedups because the time complexity of evaluating the nonlinearity in the reduced order system of equations still depends on the dimensionality of the high fidelity problem. 

For this reason, \textit{hyper-reduction} techniques have been developed to satisfy the assumption (ii) above, improving the time complexity of evaluating nonlinearities thanks to a further level of model reduction.
Among the most common hyper-reduction techniques, there are the Empirical Interpolation Methods (EIM) \cite{BarraultEmpiricalInterpolationMethod2004} and its variant the Discrete Empirical Interpolation Methods \cite{DEIM}. In DEIM, one performs an oblique projection onto a subspace of the column space of nonlinearities computed in the offline phase. This projection is chosen in such a way that only a smaller number of entries of the nonlinearity need to be computed in the online phase, and for this reason, it allows for computational cost reduction when dealing with componentwise nonlinearities. \reviewerA{For nonlocal nonlinearities, in which the components of the nonlinearity depend on every input, e.g.\ integral kernels or other integral operators like the Boltzmann collision operator \cite{BoltzmannDLR2025}%nonlinear terms such as $\operatorname{N}:\R^n\to\R^n$ defined as $ \operatorname{N}(v)_i = (Cv)_i^3$,
% \begin{equation}
% \label{eq:nonlocal-nonlinearity}
    % \operatorname{N}(v)_i = (Cv)_i^3,
% \end{equation}
%for some $C\in\R^{n\times n}$
}, 
exploiting DEIM \reviewerCommon{would} not be effective. Moreover, it can only be applied if one has access to the exact form of the nonlinearity in the system\reviewerA{, and if this nonlinearity is not programmed to be compatible with automatic differentiation \cite{AutoDiff2017}, it is not easily used in workflows requiring deep learning. This is the case, for example, in existing reduced order modeling libraries like RBniCS \cite{RozzaBallarinScandurraPichi2024}.}

In this paper, we develop a data-based hyper-reduction machine-learning technique for ROM\reviewerA{, compatible with automatic differentiation,} which we call the Neural Empirical Interpolation Method (NEIM).
The idea of NEIM is to approximate an affine decomposition of the nonlinearity in a reduced order model so that the nonlinear term can be written as the sum of the product between parameter-dependent coefficients and solution-dependent neural networks. This approximation is done with a greedy approach. Given a fixed set of parameters, NEIM first finds the parameter corresponding to the reduced order nonlinearity with highest approximate $L^2$ error. A neural network is then trained to approximate a normalized version of this nonlinearity, and the coefficient of the first term in the affine decomposition is determined by solving a linear system of equations. This is repeated iteratively, with the $L^2$ error computed relative to the current approximation of the affine decomposition. At the end of this iterative procedure, the approximations for the coefficients in the affine decomposition are interpolated from the fixed set of parameters to the whole parameter space. Because of the greedy strategy, we are able to perform a basic error analysis to investigate the performance of NEIM. Namely, we can show that the error estimate used in the algorithm is decreasing with each iteration. We also show how the final generalization error is composed of projection, training, and interpolation errors. 
NEIM has the advantage over DEIM of being a completely data-driven approach\reviewerA{, which gives NEIM the main advantage of being a far less intrusive method and more easily integrated with other deep learning techniques. Indeed, in the online phase of solving a reduced order model with NEIM, the nonlinear terms in the original full order model are never evaluated, unlike DEIM which still evaluates the nonlinearity at selected rows. Thus, the full order model code does not need to be changed to exploit \reviewerCommon{the} NEIM strategy. Moreover, unlike DEIM which involves a linear projection  and deals with componentwise nonlinearities, NEIM is a nonlinear methodology which is also suitable for the previously described nonlocal nonlinearities.} 
% \reviewerA{such as the one presented in Equation \eqref{eq:nonlocal-nonlinearity}}. 

We demonstrate the effectiveness of our methodology on solution-independent and solution-dependent nonlinearities, a nonlinear elliptic problem, and a nonlinear parabolic model of liquid crystals. In the solution-independent test case, we use NEIM to approximate a nonlinearity which depends only on the PDE parameter $\mu$ and not the PDE solution. In this case, the neural networks in the approximation should be approximately constant. In the solution-dependent test case, we modify the nonlinearity in the first example so that the neural networks should learn non-constant functions. Each of these first two methods are based on a finite difference method. The third example presents an application of NEIM to a nonlinear elliptic problem which is solved with a physics-informed neural network (PINNs) \cite{RaissiPhysicsinformedNeuralNetworks2019} in the ROM context \cite{ChenPhysicsinformedMachineLearning2021}. We then present an application of the methodology to a system of PDEs describing the dynamics of liquid crystal molecules which is discretized with finite elements and has two nonlinearities to approximate.
In addition to the advantages of the proposed methodology, NEIM can easily be used in combination with automatic differentiation in deep learning libraries like PyTorch \cite{PyTorch}. This is unlike DEIM in finite element settings in which the assembly of the nonlinearity is not easy to interface between deep learning libraries and finite element libraries. We show the compatibility of NEIM with automatic differentiation in our example with a physics-informed neural network. Furthermore, without relying on a complete approximation of the nonlinear field, NEIM provides more control and interpretability, also thanks to the developed error analysis. Lastly, these architectures are by definition small and fast to train and evaluate, allowing for their usage in standard reduction settings.

There exist several machine learning approaches to model order reduction for parameterized PDEs which tackle the hyper-reduction issue from a different perspective. These can be classified into linear and nonlinear approaches, depending on the maps exploited to approximate the solution field. As an example, neural networks \cite{HesthavenNonintrusiveReducedOrder2018,PichiArtificialNeuralNetwork2023,ChenPhysicsinformedMachineLearning2021}, interpolation \cite{Bui2003, Demo2019}, and Gaussian process regression \cite{Guo2018,Kast2020} have been used to map PDE parameters to coefficients of a linear basis for a reduced order subspace.
The focus of such works, however, is on approximating the entire parameter-solution mapping corresponding to the reduced order problem and not on designing a hyper-reduction technique to be used online.
On the other hand, nonlinear compression has gained much interest to deal with, e.g.,\ advection-dominated problems \cite{Diez2021,Taddei2020,TorloModelReductionAdvection2020,KhamlichOptimalTransportinspiredDeep2023}, and more generally to overcome the Kolmogorov barrier \cite{Ohlberger2015} by exploiting autoencoder architectures \cite{Lee2020, Romor2023, Fresca2021,PichiGraphConvolutionalAutoencoder2024,MorrisonGFNGraphFeedforward2024a}.
Since the complex parameter-solution mapping in the nonlinear setting is usually obtained through deep neural networks, the approximated map can be subject to overfitting, require many data to be trained, and lack interpretability.
The NEIM approach is also closely related to very recent works exploring the usage of empirical interpolation techniques in combination with neural networks \cite{AntilNoteDimensionalityReduction2023,CicciDeepHyROMnetDeepLearningBased2022}, and to the development of adaptive methods of constructing  suitable basis functions \cite{AmsallemNonlinearModelOrder2012,FrancoDeepOrthogonalDecomposition2024,GeelenLocalizedNonintrusiveReducedorder2022,demo2024a}.
Lastly, there are more general machine learning approaches to approximating nonlinear operators in scientific computing \cite{Kovachki2021,BoulleMathematicalGuideOperator2023}. In particular, DeepONets are a popular method used for this purpose \cite{Lu2021}. In the most basic setting, DeepONets can be thought of as having one neural network which takes in a parameter and another network which takes the evaluations of a function at various points in space. The output of the DeepONet is the dot product of the outputs of these two networks.
From this perspective, NEIM can be thought of structurally as a specific instance of DeepONet. The difference between NEIM approximations and DeepONets is the way in which NEIM is fit to data in the offline phase. NEIM relies on a greedy method of selecting parameters corresponding to individual terms in the approximation sum. In effect, it is a greedy method of producing a sequence of progressively more accurate DeepONets.

The remainder of the paper is organized as follows. Section \ref{sec:problem_setting} presents the ROM framework which motivates NEIM. Section \ref{sec:NEIM} defines the NEIM algorithm, and Section \ref{sec:numerics} presents numerical results. Finally, Section \ref{sec:conclusion} discusses conclusions.

\section{Problem Setting}
\label{sec:problem_setting}
We begin by describing our setting of reduced order models. Let $p, d\in\mathbb{N}$, and consider $\mathcal{P}\subset \R^p$ a compact parameter space and $\Omega(\mu) \subset \R^d$ a bounded spatial domain, where $\mu\in\mathcal{P}$. Then for $\mu\in\mathcal{P}$, consider the nonlinear PDE
\begin{equation}
\label{eq:continuous_pde}
\begin{split}
    \mathcal{A}(\mu)[v(x)] + \mathcal{N}(\mu)[v(x)] &= f(\mu),\quad x\in\Omega(\mu)\\
    \mathcal{B}(\mu)[v(x)] &= 0,\quad\quad x\in\partial\Omega(\mu),
\end{split}
\end{equation}
where $\mathcal{A}(\mu)$ is a linear differential operator, $\mathcal{N}(\mu)$ is a nonlinear differential operator, and $f(\mu)$ is a forcing term. $\mathcal{B}(\mu)$ is a boundary operator defined on the boundary $\partial\Omega(\mu)$ of the spatial domain. We seek a solution $v\in X$, where $X$ is some suitable function space.
To solve this problem numerically, we introduce a finite dimensional subspace $X_h \subseteq X$ in which we seek a solution to a weak formulation of Equation \eqref{eq:continuous_pde}. Using, for example, a finite element discretization of the above problem, we obtain the nonlinear system of equations (with a slight abuse of notation)
\begin{equation}
\label{eq:discrete_problem}
    A(\mu)v + \operatorname{N}(v; \mu) = f(\mu), \quad \mu\in\mathcal{P},
\end{equation}
where $A(\mu)\in\R^{n\times n}$ is the linear operator, $\operatorname{N}(\cdot; \mu):\R^n\to\R^n$ is the nonlinearity, $f(\mu)\in\R^n$ is the forcing term, and $v\in\R^n$.

\subsection{Proper Orthogonal Decomposition}
Now we describe how we perform the proper orthogonal decomposition (POD) to obtain a reduced order modeling corresponding to the discrete problem in Equation \eqref{eq:discrete_problem}. 
Suppose that we have some number $\nu$ of parameters $\mu_1^{POD},\mu_2^{POD},\dots,\mu_\nu^{POD}$. To each parameter $\mu_i^{POD}$, there corresponds a solution $v(\mu_i^{POD})$ of \eqref{eq:discrete_problem}. We collect these solutions into the snapshot matrix $S = \begin{bmatrix}v(\mu_1^{POD}) & \dots & v(\mu_\nu^{POD})\end{bmatrix}$.
The POD then consists in taking the singular value decomposition of $S = U\Sigma V^T$, where $U\in \R^{n\times n}$, $\Sigma \in \R^{n\times \nu}$, and $V\in\R^{\nu\times\nu}$. We take $U_r$ to be the matrix consisting of the first $r$ columns of $U$, and approximate the solution of Equation \eqref{eq:discrete_problem} by solving the ROM problem
\begin{equation}
    \label{eq:reduced_order_problem}
    U_r^\top A(\mu) U_r\tilde v + U_r^\top \operatorname{N}(U_r\tilde v;\mu) = U_r^\top f(\mu)
\end{equation}
so that the solution $v(\mu)$ to Equation \eqref{eq:discrete_problem} and the solution $\tilde v(\mu)$ to Equation \eqref{eq:reduced_order_problem} satisfy $v(\mu) \approx U_r\tilde v(\mu)$.
Assuming that $A(\mu)$ and $f(\mu)$ admit affine decompositions \reviewerCommon{such that}\
\[
    A(\mu) = \sum_{i=1}^{\reviewerA{q_1}} \theta_i^A(\mu) A_i,\quad f(\mu) = \sum_{i=1}^{\reviewerA{q_2}} \theta_i^f(\mu) f_i
\]
for some coefficients $\theta_i^A(\mu)$ and $\theta_i^f(\mu)$, constants $A_i\in\R^{n\times n}$ and $f_i\in\R^{n}$, \reviewerA{and $q_1,q_2\in\mathbb{N}$}, we see that the time complexity of evaluating $U_r^\top A(\mu) U_r$ and $U_r^\top f(\mu)$ is independent of $n$ and only dependent on $r$ since $U_r^\top A_i U_r$ and $U_r^\top f_i$ can be pre-computed. However, the time complexity of computing $U_r^\top \operatorname{N}(U_r\tilde v; \mu)$ still depends on $n$, so we do not see significant improvements in compute time by solving Equation \eqref{eq:reduced_order_problem} instead of Equation \eqref{eq:discrete_problem}. A typical approach to this problem is to use the discrete empirical interpolation method \cite{DEIM}. Here, we instead introduce an iterative method we call the neural empirical interpolation method in which we approximate an affine decomposition of this nonlinear term using a machine learning approach.

\subsection{Discrete Empirical Interpolation Method}
We now briefly recall the discrete empirical interpolation method (DEIM), where one approximates
\[
    \operatorname{N}(U_r\tilde v; \mu) \approx V_{k}(P^\top V_{k})^{-1}P^\top \operatorname{N}(U_r\tilde v; \mu),
\]
where $V_{k}$ is the matrix of the $k$ left singular vectors of the matrix of nonlinear snapshots of the full-order system $S_{\text{N}} = \begin{bmatrix} \operatorname{N}(v(\mu_1^{POD});\mu_1^{POD}) & \dots & \operatorname{N}(v(\mu_\nu^{POD});\mu_\nu^{POD})\end{bmatrix}$
and $P = \begin{bmatrix} e_{p_1} & \dots & e_{p_{k}}\end{bmatrix}$ is a matrix with $e_{p_i} = \begin{bmatrix} 0 & \dots & 0 & 1 & 0 & \dots & 0\end{bmatrix}^\top \in \R^n$ the $p_i$-th column of the identity matrix for $i=1,\dots,k$. The indices $p_1,\dots,p_{k}$ are determined by the greedy Algorithm \eqref{alg:DEIM}. Following the original DEIM paper \cite{DEIM}, the notation $\texttt{max}$ is MATLAB notation so that $[|\rho|, p_\ell] = \texttt{max}\{|r|\}$ implies $|\rho| = |r_{p_\ell}| = \max_{i=1,\dots,n} \{|r_i|\}$ with the smallest index taken in case of a tie.

\begin{algorithm}
\caption{DEIM}
\label{alg:DEIM}
\begin{algorithmic}[1]
\Procedure{DEIM}{$\{\operatorname{N}(v(\mu_i^{POD}); \mu_i^{POD})\}_{i=1}^\nu \subset \R^n$ linearly independent}
    \State{$[|\rho|, p_1] = \texttt{max}\{\left\lvert\operatorname{N}(v(\mu_1^{POD}); \mu_1^{POD})\right\rvert\}$}
    \State{$V_{k} = \begin{bmatrix}\operatorname{N}(v(\mu_1^{POD}); \mu_1^{POD})\end{bmatrix}$, $P = \begin{bmatrix} e_{p_1}\end{bmatrix}$}
    \For{$\ell = 2$ to $\nu$}
        \State{Solve $(P^\top V_{k})\boldsymbol{c} = P^\top \operatorname{N}(v(\mu_\ell^{POD}); \mu_\ell^{POD})$ for $\boldsymbol{c}$}
        \State{$r = \operatorname{N}(v(\mu_\ell^{POD}); \mu_\ell^{POD}) - V_{k}\boldsymbol{c}$}
        \State{$[|\rho|, p_\ell] = \texttt{max}\{|r|\}$}
        \State{$V_{k} \gets \begin{bmatrix} V_{k} & \operatorname{N}(v(\mu_\ell^{POD}); \mu_\ell^{POD})\end{bmatrix}$, $P\gets \begin{bmatrix} P & p_\ell\end{bmatrix}$}
    \EndFor
\EndProcedure
\end{algorithmic}
\end{algorithm}

\section{Neural empirical interpolation method}
\label{sec:NEIM}

Given the setup in the previous section, our goal is to develop a DEIM-inspired neural network based methodology, namely the Neural Empirical Interpolation Method (NEIM), to find an approximation of the nonlinear term of the form
\begin{equation}
    U_r^\top\operatorname{N}(U_r\tilde v;\mu) \approx \widehat{\operatorname{N}}(\tilde v;\mu) \doteq \sum_{i=1}^k \theta_i(\mu)M_{\mu^{(i)}}(\tilde v),
    \label{eq:neim_approx}
\end{equation}
where $M_{\mu^{(i)}}: \mathbb{R}^r \to \mathbb{R}^r$
is the neural network corresponding to $\mu^{(i)}$ for $i=1,\dots,k$, and $\theta_i(\mu)$ is determined through interpolation of some ``optimal'' $\theta_i$. We reinterpret the affine decomposition assumption from the neural network point of view, through a linear combination of parameter dependent coefficients $\theta_i$ and parameter independent networks $M_{\mu^{(i)}}$. Note that $\widehat{\operatorname{N}}(\tilde v;\mu) \in \R^r$, while $\operatorname{N}(U_r\tilde v; \mu) \in \R^n$. Since we only work with the truncated POD basis, in what follows, we drop the subscript $r$ in $U_r$. 

\subsection{NEIM First Step} First, we consider a finite set of parameters $\mathcal{M} = \{\mu_i\}_{i=1}^m \subset\mathcal{P}$ to which correspond the snapshots $\mathcal{S} = \{v_i = v(\mu_i)\}_{i=1}^m\subset\R^n$. We denote their projections onto the POD space by $\tilde v_i = U^\top v_i$, $i=1,\dots,m$. When computing the NEIM approximation in the ROM setting, the quantities $\mathcal{M}$ and $\mathcal{S}$ are already available from the POD approach. Then the first step in NEIM is to initialize our approximation $\widehat{\operatorname{N}}$. We greedily find new terms in the expansion of $\widehat{\operatorname{N}}$, so we denote the iterates with a superscript, the $0$-th iterate being $\widehat{\operatorname{N}}^{(0)} \equiv 0$.

Having initialized our approximation, we perform the first training step of NEIM, i.e., we find the first term in the expansion of $\widehat{\operatorname{N}}$. Because $M_{\mu^{(i)}}(\tilde v)$ in our expansion is not a function of $\mu$, we train $M_{\mu^{(1)}}$ so that $v\mapsto M_{\mu^{(1)}}(U^\top v)$ approximates the projection of the nonlinear term $U^\top N(\cdot; \mu^{(1)})$ for some $\mu^{(1)}$. 
We choose $\mu^{(1)}$ so that a measurement of error is maximized:
\begin{equation*}
    \mu^{(1)} = \argmax_{\mu\in\mathcal{M}} \sum_{i=1}^m w_e(\mu_i;\mu)\mathcal{E}^{(0)}(v_i; \mu)^2,
\end{equation*}
where $\mathcal{E}^{(0)}(v; \mu) = \|U^\top\operatorname{N}(v;\mu) - \widehat{\operatorname{N}}^{(0)}(U^\top v;\mu)\|_2$.
In the above expression, $w_e(\mu_i;\mu)\in [0,\infty)$ are ``error" weights which can be interpreted as quadrature weights. In other words, for a fixed $\mu\in\mathcal{M}$, the sum above is an approximation of the $L^2$ error over $\mathcal{P}$ between $U^\top\operatorname{N}(v(\cdot);\mu)$ and its approximation $\widehat{\operatorname{N}}^{(0)}(U^\top v(\cdot); \mu)$, where $v(\cdot)$ denotes the mapping from a parameter $\mu$ to its corresponding high fidelity solution denoted $v(\mu)$. Thus $\mu^{(1)}$ is the parameter such that the corresponding $L^2$ error approximation is maximized. Note that by choosing the weights appropriately, we can also obtain a more local error estimator, e.g., by choosing $w_e(\mu_i; \mu) = \delta_{\mu_i\mu}$, the Kronecker delta.
Having chosen $\mu^{(1)}$, we train $M_{\mu^{(1)}}$ so that $M_{\mu^{(1)}}(\tilde v_i) \approx U^\top \operatorname{N}(v_i;\mu^{(1)}) / \|U^\top \operatorname{N}(v_i;\mu^{(1)})\|_2$ for $v_i\in\mathcal{S}$. 
The main reason for evaluating the nonlinearity for a fixed parameter $\mu^{(1)}$ while varying the solution $v_i$ \reviewerCommon{comes} from viewing the nonlinear term as an operator; thus the NEIM expansion we obtain approximates the interpolation in $\mu$ of an operator evaluated at particular parameters. To find this first ``mode" of the expansion, we take $M_{\mu^{(1)}}$ to be a simple feedforward neural network trained by minimizing the weighted squared loss
\begin{equation}\label{eq:network_first_mode}
    \sum_{i=1}^m w_t^{(1)}(\mu_i)\left\|M_{\mu^{(1)}}(\tilde v_i) - \frac{U^\top \operatorname{N}(v_i;\mu^{(1)})}{\|U^\top \operatorname{N}(v_i;\mu^{(1)})\|_2}\right\|_2^2
\end{equation}
with respect to the network parameters, where $w_t^{(1)}(\mu_i)\in[0,\infty)$ for $i=1,\dots,m$ determine the relative importance of each sample in the training data for the optimization of the neural network ``training" weights.

At this point we need to find the $\mu$-dependent scaling factor for $M_{\mu^{(1)}}$. Thus, we seek some $\theta_1^{(1)}(\mu)$ so that $\theta_1^{(1)}(\mu)M_{\mu^{(1)}}(\tilde v_i)$ is ``close" to $U^\top\operatorname{N}(v_i;\mu)$ for $\mu\in\mathcal{M}$, and $v_i\in\mathcal{S}$. More precisely, we perform an optimization task to find $\theta_1^{(1)}(\mu)$ minimizing 
\begin{equation}
\label{eq:theta1_optimization}
    \sum_{\mu\in\mathcal{M}}\sum_{i=1}^m w_e(\mu_i;\mu)\| U^\top\operatorname{N}(v_i; \mu) - \theta_1^{(1)}(\mu)M_{\mu^{(1)}}(\tilde v_i)\|_2^2 .
\end{equation}
The new approximation of the projected nonlinearity is $\widehat{\operatorname{N}}^{(1)}(\tilde v;\mu) = \theta_1^{(1)}(\mu)M_{\mu^{(1)}}(\tilde v)$, $\tilde v\in\R^r$. Note that $\theta_1^{(1)}(\mu)$ is only defined for $\mu\in\mathcal{M}$ because the $\theta$ coefficients are interpolated from $\mathcal{M}$ to all of $\mathcal{P}$ at the end of the NEIM approximation procedure.
We can find the value of $\theta_1^{(1)}(\mu)$ by setting the partials of \eqref{eq:theta1_optimization} equal to 0:
\begin{align*}
    &-2\sum_{i=1}^m w_e(\mu_i;\mu)\left\langle M_{\mu^{(1)}}(\tilde v_i), U^\top\operatorname{N}(v_i; \mu) - \theta_1^{(1)}(\mu)M_{\mu^{(1)}}(\tilde v_i)\right\rangle = 0\\
    &\implies \theta_1^{(1)}(\mu) = \frac{\sum_{i=1}^m w_e(\mu_i;\mu)\langle M_{\mu^{(1)}}(\tilde v_i), U^\top\operatorname{N}(v_i; \mu)\rangle}{\sum_{i=1}^m w_e(\mu_i;\mu)\|M_{\mu^{(1)}}(\tilde v_i)\|_2^2},\quad \mu\in\mathcal{M}.
\end{align*}
When all of the weights are equal, having trained $M_{\mu^{(1)}}(\tilde v_i)$ so that $M_{\mu^{(1)}}(\tilde v_i) \approx \frac{U^\top \operatorname{N}(v_i;\mu^{(1)})}{\|U^\top\operatorname{N}(v_i;\mu^{(1)})\|_2}$, we expect that $\theta_1^{(1)}(\mu^{(1)}) \approx \frac1m\sum_{i=1}^m\|U^\top \operatorname{N}(v_i;\mu^{(1)})\|_2$. Indeed, taking the weights $w_e \equiv C$ gives 
\[
    \theta_1^{(1)}(\mu^{(1)}) \approx \frac{\sum_{i=1}^m C \left\langle \frac{U^\top\operatorname{N}(v_i;\mu^{(1)})}{\|U^\top\operatorname{N}(v_i;\mu^{(1)})\|_2}, U^\top\operatorname{N}(v_i;\mu^{(1)})\right\rangle}{\sum_{i=1}^m C\left\|\frac{U^\top\operatorname{N}(v_i;\mu^{(1)})}{\|U^\top\operatorname{N}(v_i;\mu^{(1)})\|_2}\right\|_2^2} = \frac1m \sum_{i=1}^m \|U^\top\operatorname{N}(v_i;\mu^{(1)})\|_2.
\]
Now, the approximation $\widehat{\operatorname{N}}^{(1)}(\tilde v;\mu)$ should be accurate for $\mu$ close to $\mu^{(1)}$ and $\tilde v$ close to $U^\top v$ with $v\in\mathcal{S}$. The first step of NEIM is summarized in Algorithm \eqref{alg:NEIM_initialization}.

\begin{algorithm}
\caption{NEIM First Step}
\label{alg:NEIM_initialization}
\begin{algorithmic}[1]
\Procedure{NEIM\_Init}{$\mathcal{M}$, $\mathcal{S}$}
    \State{$\widehat{\operatorname{N}}^{(0)} \gets 0$}
    \State{$\mu^{(1)} \gets \argmax_{\mu\in\mathcal{M}} \sum_{i=1}^m w_e(\mu_i;\mu)\mathcal{E}^{(0)}(v_i; \mu)^2$}
    \State{Train $M_{\mu^{(1)}}$ so that $M_{\mu^{(1)}}(\tilde v_i) \approx U^\top\operatorname{N}(v_i; \mu^{(1)})/\|U^\top \operatorname{N}(v_i;\mu^{(1)})\|_2$ for all $v_i\in\mathcal{S}$}
    \State{Find $\theta_1^{(1)}(\mu)$ such that $\sum_{\mu\in\mathcal{M}}\sum_{i=1}^m w_e(\mu_i;\mu)\| U^\top\operatorname{N}(v_i; \mu) - \theta_1^{(1)}(\mu)M_{\mu^{(1)}}(\tilde v_i)\|_2^2 = \min!$}
    \State{$\widehat{\operatorname{N}}^{(1)}(\tilde v;\mu) := \theta_1^{(1)}(\mu)M_{\mu^{(1)}}(\tilde v)$, $\tilde v\in\R^r,\mu\in\mathcal{M}$}
\EndProcedure
\end{algorithmic}
\end{algorithm}

\subsection{NEIM Update Step}
To add a term to the NEIM expansion, there are two main modifications of the first step described above. First, we must be careful to orthogonalize our training data with respect to the previously trained neural networks in the NEIM expansion to avoid instability in the estimates of the coefficients $\theta_i(\mu)$. Second, to solve for the coefficients in the NEIM expansion, we must solve a linear system. Indeed, the first step of NEIM can be viewed as solving a $1\times 1$ linear system for the coefficients. In addition to describing these modifications here, we also generalize the other steps of the NEIM update.
At iteration $j$ of the procedure, we have already computed $\widehat{\operatorname{N}}^{(j-1)}$, and  we want to update this approximation by adding a term to the NEIM expansion. We begin, as before, by finding the parameter for which the approximation $\widehat{\operatorname{N}}^{(j-1)}$ has the worst error:
\begin{equation}
\label{eq:error_expression}
\begin{split}
    \mu^{(j)} &= \argmax_{\mu\in\mathcal{M}\setminus\{\mu^{(1)},\dots,\mu^{(j-1)}\}} \sum_{i=1}^m w_e(\mu_i;\mu)\mathcal{E}^{(j-1)}(v_i;\mu)^2,\\
    &\quad\quad\text{where}\quad \mathcal{E}^{(j-1)}(v;\mu)^2 = \| U^\top\operatorname{N}(v; \mu) - \widehat{\operatorname{N}}^{(j-1)}(U^\top v;\mu)\|_2^2.
\end{split}
\end{equation}
By removing previously selected values for $\mu$ from $\mathcal{M}$ in \eqref{eq:error_expression}, we avoid \reviewerCommon{reselecting} the same parameters, due to\reviewerCommon{,} e.g.\reviewerCommon{,} a poorly trained corresponding network, possibly causing stability issues. Moreover, there is no reason why a network with a re-selected parameter would be trained any better than the first network corresponding to that parameter. \reviewerB{For example, while it is expected that parameters selected by NEIM will correspond to the smallest approximation errors, the results of the neural network training depend on the weights $w_t^{(j)}$, and some choices of these weights (e.g., choosing $w_t^{(j)}(\mu^{(j)})=0$) can lead to results with the error at the selected parameter not decreasing much after training compared to the error at other parameters.
Now, having} chosen $\mu^{(j)}$, we want to train $M_{\mu^{(j)}}$ so that 
$$M_{\mu^{(j)}}(\tilde v_i) \approx U^\top \operatorname{N}(v_i;\mu^{(j)}) / \|U^\top \operatorname{N}(v_i;\mu^{(j)})\|_2, \quad \text{for} \quad v_i\in\mathcal{S}.$$
To avoid training instability, $M_{\mu^{(j)}}$ is normalized and orthogonalized with respect to the previously trained neural networks. We define $y^{(j)}(\mu_i) = U^\top\operatorname{N}(v_i; \mu^{(j)})$ for $i=1,\dots,m$, and let $z^{(j)}(\mu_i)$ be the result of performing Gram-Schmidt on $y^{(j)}(\mu_i)$ \reviewerCommon{with respect to} the networks  $M_{\mu^{(1)}}(\tilde v_i), \dots, M_{\mu^{(j-1)}}(\tilde v_i)$. We describe this in Algorithm \eqref{alg:gram_schmidt}. 
\reviewerA{\begin{remark}
Note that the maximum number of terms in the NEIM approximation is $r$ if we assume that the neural networks train well. Indeed, in this case, $\{M_{\mu^{(1)}}(\widetilde{v}_i),\dots,M_{\mu^{(r)}}(\widetilde{v}_i)\}$ would approximate a basis of $\R^r$ for each $i=1,\dots,m$, and any additional networks in the NEIM approximation would not greatly reduce the approximation error. This can be compared with DEIM which typically interpolates at many more points than $r$ (with the number of points bounded by the high fidelity dimension). This difference is due to the fact that DEIM approximates a projection of $\operatorname{N}$ in the high fidelity space which is then multiplied by $U^\top$, but NEIM directly approximates $U^\top\operatorname{N}$ in the reduced order space.
\end{remark}}
\begin{algorithm}
\caption{Output Data Orthogonalization}
\label{alg:gram_schmidt}
\begin{algorithmic}[1]
\Procedure{Gram\_Schmidt}{$y^{(j)}(\mu_i),M_{\mu^{(1)}}(\tilde v_i), \dots, M_{\mu^{(j-1)}}(\tilde v_i)$}
    \State{$z^{(j)}(\mu_i) \gets y^{(j)}(\mu_i)$}
    \For{$k=1,\dots,j-1$}
        \State{$z^{(j)}(\mu_i) \gets z^{(j)}(\mu_i) - \frac{\langle z^{(j)}(\mu_i), M_{\mu^{(k)}}(\tilde v_i)\rangle}{\|M_{\mu^{(k)}}(\tilde v_i)\|_2^2}M_{\mu^{(k)}}(\tilde v_i)$}
    \EndFor
    \State{$z^{(j)}(\mu_i) \gets z^{(j)}(\mu_i)/\|z^{(j)}(\mu_i)\|_2$}
\EndProcedure
\end{algorithmic}
\end{algorithm}

\noindent We train $M_{\mu^{(j)}}$ with the orthogonalized data by minimizing the weighted squared loss
\begin{equation}
\label{eq:network_j_mode}
    \sum_{i=1}^m w_t^{(j)}(\mu_i)\left\|M_{\mu^{(j)}}(\tilde v_i) - z^{(j)}(\mu_i)\right\|_2^2
\end{equation}
\reviewerCommon{with respect to the network parameters,} where the weights $w_t^{(j)}(\mu_i)\in[0,\infty)$ for $i=1,\dots,m$ have the same meaning as in Equation \eqref{eq:network_first_mode}.
After training $M_{\mu^{(j)}}$, we find $\theta_1^{(j)}(\mu),\dots,\theta_j^{(j)}(\mu)$ for $\mu\in\mathcal{M}$ such that 
\begin{equation}
\label{eq:theta_minimization_problem}
    \sum_{\mu\in\mathcal{M}}\sum_{i=1}^m w_e(\mu_i;\mu)\left\|U^\top\operatorname{N}(v_i;\mu) - \sum_{\ell=1}^j \theta_\ell^{(j)}(\mu)M_{\mu^{(\ell)}}(\tilde v_i)\right\|_2^2
\end{equation}
is minimized. This is equivalent to minimizing for each fixed $\mu\in\mathcal{M}$ the quantity
\[
    \sum_{i=1}^m w_e(\mu_i;\mu)\left\|U^\top\operatorname{N}(v_i;\mu) - \sum_{\ell=1}^j \theta_\ell^{(j)}(\mu)M_{\mu^{(\ell)}}(\tilde v_i)\right\|_2^2.
\]
 This problem amounts to solving the $j\times j$ linear system resulting from setting the partial derivatives of this expression equal to zero. Doing so for $\mu\in\mathcal{M}$, we obtain
\[
    -2\sum_{i=1}^m w_e(\mu_i;\mu)\left\langle M_{\mu^{(k)}}(\tilde v_i), U^\top\operatorname{N}(v_i;\mu) - \sum_{\ell=1}^j \theta_\ell^{(j)}(\mu)M_{\mu^{(\ell)}}(\tilde v_i)\right\rangle = 0,\quad k=1,\dots,j,
\]
which results in the system $\textsf{A}\theta = \textsf{b}$ for $\theta = \left[\theta_1^{(j)}(\mu), \dots, \theta_j^{(j)}(\mu)\right]^\top$, with
\begin{equation}
\label{eq:theta_system}
\begin{split}
    \textsf{A}_{k\ell} &= \sum_{i=1}^m w_e(\mu_i;\mu)\left\langle M_{\mu^{(k)}}(\tilde v_i), M_{\mu^{(\ell)}}(\tilde v_i)\right\rangle,\\
    \textsf{b}_k &= \sum_{i=1}^m w_e(\mu_i;\mu)\left\langle M_{\mu^{(k)}}(\tilde v_i), U^\top\operatorname{N}(v_i;\mu)\right\rangle,\quad k,l=1,\dots,j.
\end{split}
\end{equation}
We highlight the fact that the optimization problem defined by the quantity in \eqref{eq:theta_minimization_problem} depends quadratically on $m$ and involves solving $m$ linear systems of size $j\times j$. After solving this system for each $\mu\in\mathcal{M}$, we define the $j$-th approximation as $$\widehat{\operatorname{N}}^{(j)}(\tilde v;\mu) = \sum_{\ell=1}^j \theta_\ell^{(j)}(\mu)M_{\mu^{(\ell)}}(\tilde v)\quad \text{for}\quad \tilde v\in\R^r, \ \mu\in\mathcal{M}.$$
When the algorithm meets the termination condition given by a suitable stopping criteria (e.g.\reviewerCommon{,} the error quadrature \eqref{eq:error_expression} is sufficiently small), we either train a neural network or use interpolation to predict the vector of $\theta_\ell$ for a newly given parameter $\mu$. The method can be chosen depending on the preferred regression or interpolation context.
The full NEIM training procedure is described in Algorithm \eqref{alg:NEIM_full}. We postpone the discussion regarding weights and stopping criteria to the next section.

\begin{algorithm}
\caption{NEIM Training}
\label{alg:NEIM_full}
\begin{algorithmic}[1]
\Procedure{NEIM\_Train}{$\mathcal{M}$, $\mathcal{S}$}
    \State{$\widehat{\operatorname{N}}^{(1)}, \mu^{(1)} \gets \textproc{NEIM\_Init}(\mathcal{M}, \mathcal{S})$}
    \State{$j \gets 2$}
    \While{stopping criteria not met}
        \State{$\mu^{(j)} \gets \arg\max_{\mu\in\mathcal{M}\setminus\{\mu^{(1)},\dots,\mu^{(j-1)}\}} \sum_{i=1}^m w_e(\mu_i;\mu)\mathcal{E}^{(j-1)}(v_i;\mu)^2$}
        \For{$i=1,\dots,m$}
            \State{$y^{(j)}(\mu_i) \gets U^\top\operatorname{N}(v_i; \mu^{(j)})$}
            \State{$z^{(j)}(\mu_i) \gets \textproc{Gram\_Schmidt}(y^{(j)}(\mu_i), M_{\mu^{(1)}}(\tilde v_i), \dots, M_{\mu^{(j-1)}}(\tilde v_i))$}
        \EndFor
        \State{Train $M_{\mu^{(j)}}$ by minimizing $\sum_{i=1}^m w_t^{(j)}(\mu_i)\left\|M_{\mu^{(j)}}(U^\top v(\mu_i)) - z^{(j)}(\mu_i)\right\|^2$}
        \State{Define $\textsf{A}(\mu), \textsf{b}(\mu)$ as in \eqref{eq:theta_system}}
        \State{$\begin{bmatrix}\theta_1^{(j)}(\mu) & \dots & \theta_j^{(j)}(\mu)\end{bmatrix}^\top \gets \textsf{A}(\mu)^{-1}\textsf{b}(\mu)$ for $\mu\in\mathcal{M}$}
        \State{$\widehat{\operatorname{N}}^{(j)}(\tilde v;\mu) := \sum_{\ell=1}^j \theta_\ell^{(j)}(\mu)M_{\mu^{(\ell)}}(\tilde v)$ for $\tilde v\in\R^r,\mu\in\mathcal{M}$}
        \State{$j\gets j+1$}
    \EndWhile
    \State{Perform regression/interpolation for $\theta_\ell$}
\EndProcedure
\end{algorithmic}
\end{algorithm}

\begin{remark}
    In the NEIM algorithm, we exploit the weighting of the least squares problems in two different steps. Indeed, we weight the training data in the neural network loss at each iteration (Equation \eqref{eq:network_j_mode}), and we also weight the least squares problem for finding the $\theta(\mu)$ coefficients (Equation \eqref{eq:theta_minimization_problem}). Of these, the weighting of the training data is most important practically speaking. To obtain good results, it is needed that the trained neural networks $M_{\mu^{(j)}}$ \reviewerCommon{are} accurate when evaluated at \reviewerCommon{$\tilde{v}(\mu^{(j)})$} and $\mu^{(j)}$. Towards this goal, the weight corresponding to \reviewerCommon{$\tilde{v}(\mu^{(j)})$} has to be much  larger relative to the other weights.
\end{remark}
\begin{remark}
    Assuming that the neural networks trained well, due to the pointwise orthogonalization, we would expect the matrix $\textsf{A}(\mu)$ to be close in norm to a diagonal matrix. Moreover, if they are exactly orthogonal pointwise (which is difficult to observe in practice), then $\textsf{A}(\mu)$ is diagonal with
    \[
        \textsf{A}(\mu)_{kk} = \sum_{i=1}^m w_e(\mu_i;\mu)\|M_{\mu^{(k)}}(\tilde v_i)\|_2^2,
    \]
    so that, solving the system $\textsf{A}(\mu)\theta = \textsf{b}(\mu)$ exactly yields
    \[
        \theta_k^{(j)}(\mu) = \frac{\sum_{i=1}^m w_e(\mu_i;\mu)\left\langle M_{\mu^{(k)}}(\tilde v_i), U^\top\operatorname{N}(v_i;\mu)\right\rangle}{\sum_{i=1}^m w_e(\mu_i;\mu)\|M_{\mu^{(k)}}(\tilde v_i)\|_2^2}.
    \]
    Consequently, this gives
    \[
        \widehat{\operatorname{N}}^{(j)}(\tilde v;\mu) = \sum_{k=1}^j \frac{\sum_{i=1}^m w_e(\mu_i;\mu)\left\langle M_{\mu^{(k)}}(\tilde v_i), U^\top\operatorname{N}(v_i;\mu)\right\rangle}{\sum_{i=1}^m w_e(\mu_i;\mu)\|M_{\mu^{(k)}}(\tilde v_i)\|_2^2}M_{\mu^{(k)}}(\tilde v).
    \]
    Interpreting the sums in the coefficient $\theta_k^{(j)}$ as quadrature rules for an $L^2$ inner product, we notice that $\widehat{\operatorname{N}}^{(j)}$ is the sum of approximate projections of $U^\top\operatorname{N}(v;\mu)$ onto the orthogonal neural networks $M_{\mu^{(k)}}$, $k=1,\dots,j$.
\end{remark}

\subsection{Choice of the Weights and Stopping Criteria}
\label{sec:weights_and_stopping}
There are many different ways of defining stopping criteria or the most appropriate weighting rules. One could exploit the knowledge of the system and its behavior, or the rule of thumb given in the first remark, but in general this is problem specific (e.g., dependent on the number of snapshots available, the complexity of the solution as parameters vary, etc.). Let us now describe possible ways to choose the weights for the least square procedures. The error weights $w_e(\mu_i;\mu)$ we choose to find the $\theta(\mu)$ coefficients can be interpreted as quadrature weights.
We can also interpret training weights $w_t^{(j)}(\mu_i)$ as quadrature weights, but because we want the neural network architectures to be small so that their evaluation is fast in the online phase, it is not necessarily best to use ``optimal'' quadrature weights. For example, if the parameters $\mu_i$ are equi-spaced, then when the neural network is too small and all training points are weighted equally, the network has difficulty fitting any of the data at all. Instead, it may be preferable to give nonzero weights to only a small number of parameters $\mu_i$. For example, one could set $w_t^{(j)}(\mu_i) := \reviewerCommon{\delta_{\mu_i\mu^{(j)}}}$ to localize the network training. Another option would be a smooth approximation of the Kronecker delta, for some scalar parameters $C,\zeta>0$
\[
    w_t^{(j)}(\mu_i) := C\operatorname{exp}\left(-\zeta\|\mu_i - \reviewerCommon{\mu^{(j)}}\|_2^2\right).
\]
 In the numerical experiments, we specify the weights in each example.
Let us now focus on the different choices to implement stopping criteria for the NEIM algorithm. Within our iterative methods, as we will show in the next section, a measure of the error is provably non-increasing for each subsequent iteration, thus, we can stop NEIM whenever this error is sufficiently small. 
However, we only know that this error is non-increasing, so it is possible that a desired error tolerance is never reached, so this should be combined with a maximum number of iterations. One could also choose to stop NEIM when this error stops decreasing significantly after each iteration, as in the Elbow Method commonly used in clustering algorithms \cite{elbow}.

\subsection{Error Analysis}
We begin by showing that the measure of error we designed for greedily selecting parameters is non-increasing. We then discuss a local error estimate for NEIM to better identify the different contributions to the error.

\begin{proposition}[Error is non-increasing]
\label{prop:error-nonincreasing}
    The function given by
    \[
        j\mapsto \sum_{i=1}^m w_e(\mu_i;\mu)\mathcal{E}^{(j-1)}(v_i;\mu)^2,\quad j\in\mathbb{N}
    \]
    is non-increasing for each $\mu\in\mathcal{M}$. In particular, if $w_e(\mu_i;\mu) = \delta_{\mu_i\mu}$, we have that
        $j\mapsto\mathcal{E}^{(j-1)}(v_i;\mu_i),\ j\in\mathbb{N}$
    is non-increasing.
\end{proposition}
\begin{proof}
    At each iteration of NEIM, for every $\mu\in\mathcal{M}$, $\theta_1^{(j)}(\mu),\dots,\theta_j^{(j)}(\mu)$ minimize
    \[\sum_{i=1}^m w_e(\mu_i;\mu)\mathcal{E}^{(j)}(v_i;\mu)^2 = 
        \sum_{i=1}^m w_e(\mu_i;\mu)\left\|U^\top\operatorname{N}(v_i;\mu) - \sum_{\ell=1}^j \theta_\ell^{(j)}(\mu)M_{\mu^{(\ell)}}(\tilde v_i)\right\|_2^2.
    \]
    Thus, we have that
    \begin{align*}
        \sum_{i=1}^m w_e(\mu_i;\mu)&\mathcal{E}^{(j)}(v_i;\mu)^2 
        % = \sum_{i=1}^m w_e(\mu_i;\mu)\left\|U^\top\operatorname{N}(v_i;\mu) - \sum_{\ell=1}^j \theta_\ell^{(j)}(\mu)M_{\mu^{(\ell)}}(\tilde v_i)\right\|_2^2\\ 
        % &
        = \argmin_{\theta\in\R^j} \sum_{i=1}^m w_e(\mu_i; \mu)\left\|U^\top \operatorname{N}(v_i;\mu) - \sum_{\ell=1}^j \theta_\ell M_{\mu^{(\ell)}}(\tilde v_i)\right\|_2^2\\
        &\le \sum_{i=1}^m w_e(\mu_i;\mu)\left\|U^\top\operatorname{N}(v_i;\mu) - \sum_{\ell=1}^{j-1} \theta_\ell^{(j-1)}(\mu)M_{\mu^{(\ell)}}(\tilde v_i) - 0\cdot M_{\mu^{(j)}}(\tilde v_i)\right\|_2^2\\ 
        &= \sum_{i=1}^m w_e(\mu_i;\mu)\mathcal{E}^{(j-1)}(v_i;\mu)^2,
    \end{align*}
    where in the inequality we choose $\theta = (\theta_1^{(j-1)}, \dots, \theta_{j-1}^{(j-1)}, 0) \in \R^j$. 
\end{proof}
\begin{remark}
    Examining the proof of the previous proposition carefully, we see that it can be generalized for more general error measures of the form 
    \[
        h_\mu^{(j)}(\theta_1^{(j)},\dots,\theta_j^{(j)}) := g\left(\sum_{\ell=1}^j \theta_\ell^{(j)}M_{\mu^{(\ell)}}(\tilde v_1),\dots,\sum_{\ell=1}^j \theta_\ell^{(j)}M_{\mu^{(\ell)}}(\tilde v_m), \mu\right),
    \]
    where $h_\mu^{(j)}:\R^j\to\R$ has a minimum. Then, the expression in Equation \eqref{eq:error_expression} becomes
    \[
        \mu^{(j)} = \argmax_{\mu\in\mathcal{M}\setminus\{\mu^{(1)},\dots,\mu^{(j-1)}\}} h_\mu^{(j-1)}(\theta_1^{\reviewerCommon{(j-1)}}(\mu),\dots,\theta_{j-1}^{(j-1)}(\mu)),
    \]
    and the expression in equation \eqref{eq:theta_minimization_problem} is replaced by
    \[
        \sum_{\mu\in\mathcal{M}} h_\mu^{(j)}(\theta_1^{(j)}(\mu),\dots,\theta_j^{(j)}(\mu)),
    \]
    where the $\theta_\ell^{(j)}(\mu)$ are to-be-determined by the subsequent steps of NEIM.
\end{remark}
\begin{remark}
    The previous proposition still holds if we hold previously computed $\theta$'s constant in the NEIM algorithm. Doing so would not require solving a linear system at each iteration of the NEIM algorithm, but since this linear system is small, we can cheaply recompute it during each iteration.
\end{remark}

\subsubsection{Local Error Estimate}
We now seek a theoretical estimate to explain the sources of error in NEIM. For simplicity, we will consider the case in which the NEIM networks approximate constant vectors, that is, $w_t^{(j)}(\mu_i) > 0$ if and only if $\mu_i = \mu^{(j)}$, and that $w_e(\mu_i; \mu_j) = \delta_{ij}$.
We define the matrix of exact vectors as 
$M^{(k)} = \begin{bmatrix} z^{(1)}(\mu^{(1)}) & \dots & z^{(k)}(\mu^{(k)})\end{bmatrix}$, and the matrix of its approximation by neural networks as $\widehat{M}^{(k)}(\mu) = \begin{bmatrix}M_{\mu^{(1)}}(\tilde v(\mu)) & \dots & M_{\mu^{(k)}}(\tilde v(\mu))\end{bmatrix}$.
% \[
%     M^{(k)} = \begin{bmatrix}
%     \vline & & \vline\\
%        z^{(1)}(\mu^{(1)}) & \dots & z^{(k)}(\mu^{(k)})\\
%     \vline & & \vline
%     \end{bmatrix}, \quad \text{and} \quad
%     \widehat{M}^{(k)}(\mu) = \begin{bmatrix}
%         \vline & & \vline\\
%         M_{\mu^{(1)}}(\tilde v(\mu)) & \dots & M_{\mu^{(k)}}(\tilde v(\mu))\\
%         \vline & & \vline
%     \end{bmatrix}.
% \]
Now define the vector of interpolated $\theta$ coefficients resulting from NEIM to be
\[
    \widehat\theta^{(k)}(\mu) = (\widehat\theta_1^{(k)}(\mu),\dots,\widehat\theta_k^{(k)}(\mu)),
\]
and the exact $\theta$ coefficients to be
\[
    \theta^{(k)}(\mu) = \argmin_{\theta\in\R^{k}} \|\operatorname{N}(\mu) - M^{(k)}\theta\|_2,\quad \mu\in\mathcal{P}.
\]
With a little abuse of notation, we denote $\operatorname{N}(\mu) := U^\top\operatorname{N}(v(\mu); \mu)$, and we want to estimate
\[
    \mathcal{E}_{\text{NEIM}}^{(k)} := \|\operatorname{N} - \widehat{M}^{(k)}\widehat{\theta}^{(k)}\|_{L^2(\mathcal{P})}.
\]
Repeatedly applying the triangle inequality and Cauchy-Schwarz inequality, we have
\begin{align*}
    \mathcal{E}_{\text{NEIM}}^{(k)}
        &\le \|\operatorname{N} - M^{(k)}\theta^{(k)}\|_{L^2} + \|M^{(k)}\theta^{(k)} - \widehat{M}^{(k)}\theta^{(k)}\|_{L^2} + \|\widehat{M}^{(k)}\theta^{(k)} - \widehat{M}^{(k)}\widehat{\theta}^{(k)}\|_{L^2}\\
        &\le \|\operatorname{N} - M^{(k)}\theta^{(k)}\|_{L^2} + \sum_{i=1}^k \|z^{(i)}(\mu^{(i)}) - M_{\mu^{(i)}}(\tilde v(\cdot))\|_{L^2}\|\theta_i^{(k)}\|_{\reviewerCommon{L^\infty}} \\
        &\hspace{1ex} + \sum_{i=1}^k \|M_{\mu^{(i)}}(\tilde v(\cdot))\|_{L^2}\|\theta_i^{(k)} - \widehat{\theta}_i^{(k)}\|_{\reviewerCommon{L^\infty}}\\
        &\le \underbrace{\|\operatorname{N} - M^{(k)}\theta^{(k)}\|_{L^2}}_{\text{Projection error}} + \sum_{i=1}^k \underbrace{\|z^{(i)}(\mu^{(i)}) - M_{\mu^{(i)}}(\tilde v(\cdot))\|_{L^2}}_{\text{Training error}}\|\theta_i^{(k)}\|_{\reviewerCommon{L^\infty}} \\
        &\hspace{1ex} + \lvert\mathcal{P}\rvert^{1/2}\sum_{i=1}^k \underbrace{\|\theta_i^{(k)} - \widehat{\theta}_i^{(k)}\|_{\reviewerCommon{L^\infty}}}_{\text{Interpolation error}} + \sum_{i=1}^k \underbrace{\|M_{\mu^{(i)}}(\tilde v(\cdot)) - z^{(i)}(\mu^{(i)})\|_{L^2}\|\theta_i^{(k)} - \widehat{\theta}_i^{(k)}\|_{\reviewerCommon{L^\infty}}}_{\text{Training/interpolation error interaction}},
\end{align*}
where we used the triangle inequality and the fact that $\|z^{(i)}(\mu^{(i)})\|_2 = 1$.

Note that the projection error above is decreasing to $0$ in $k$ if the nonlinearities used for training NEIM are linearly independent. The training error depends on the neural network architectures used, the density of $\mathcal{M}$ in $\mathcal{P}$, and the network training procedure. Finally, the interpolation error also depends on the density of $\mathcal{M}$ in $\mathcal{P}$.

\subsection{Comparison with State-of-the-Art Techniques}
Here, we qualitatively compare NEIM and the discrete empirical interpolation method (DEIM) \cite{DEIM} to highlight a few advantages of NEIM. We then briefly relate NEIM to DeepONets \cite{Lu2021}. 

In NEIM, we build an approximation of a nonlinearity by greedily choosing parameters corresponding to large error. This is in contrast to DEIM in which one computes a selection operator for the rows of the nonlinearity, and DEIM interpolates exactly at these selected rows. Thus, NEIM allows adaptivity in parameter space, unlike DEIM which is more aptly described as spatial interpolation. NEIM's adaptivity in parameter space would be useful if a given region of parameter space requires more accurate approximation than another region of parameter space, for example.

Now for both NEIM and DEIM, if the nonlinearity acts componentwise on its input, then the complexity of evaluating the nonlinearity depends only on the reduced dimension $r$. However, DEIM is not as efficient if the nonlinearity does not act componentwise on its input. NEIM, on the other hand, allows efficient computation, even for nonlinearities which do not act componentwise. This is due to the fact that DEIM performs exact evaluation of the nonlinearity at selected rows, unlike NEIM which does not rely on the explicit form of the nonlinearity, instead approximating simultaneously both the nonlinearity and a nonlinear projection of the nonlinearity. 

Because NEIM does not rely on the explicit form of the nonlinearity and instead only data, it is also possible to use NEIM in a strictly data-driven context without explicitly evaluating the nonlinearity. This is not possible with DEIM.

It should be noted, however, that in cases where both NEIM and DEIM can be used to efficiently reduce the time complexity of evaluating nonlinearities, DEIM will typically exhibit faster decay of error, though there are subtleties in this regard (see the numerical results for further discussion of this point). 
This comparison is summarized in Table \ref{table:NEIM-DEIM-comparison}.

\begin{table}[h]
\centering
\caption{Properties of the Neural Empirical Interpolation Method (NEIM) and the Discrete Empirical Interpolation Method (DEIM).}
\resizebox{\textwidth}{!}{\begin{tabular}{|c|c|}
    \hline
     \cellcolor[HTML]{E5E3E3}\textbf{NEIM} & \cellcolor[HTML]{E5E3E3}\textbf{DEIM} \\
     \hline
     Greedy in parameter space & Greedy in spatial domain\\
     No reliance on explicit form of nonlinearity & Exact evaluation of nonlinearity at selected rows\\
     Nonlinear projection & Linear projection\\
     Efficient for nonlocal and local nonlinearities & Efficient for only local nonlinearities\\
     \hline
\end{tabular}}
\label{table:NEIM-DEIM-comparison}
\end{table}

We now briefly relate NEIM to DeepONets. DeepONets are a specific neural network architecture designed to approximate nonlinear operators \cite{Lu2021}. NEIM is related to DeepONets in that NEIM is essentially just a vector-valued DeepONet structured and trained in a greedy fashion when the $\theta$ coefficients are predicted by a neural network. Indeed, the $\theta$ coefficients in NEIM correspond to the trunk net of a DeepONet, while the vector-valued neural networks in the NEIM expansion correspond to the branch nets of a stacked DeepONet. A sketch of the architecture is given in Figure \ref{fig:deeponet-diagram}. This suggests that analysis of DeepONets may be modified to analyze NEIM as well.

\begin{figure}[bht]
    \centering
    \includegraphics[width=0.71\textwidth]{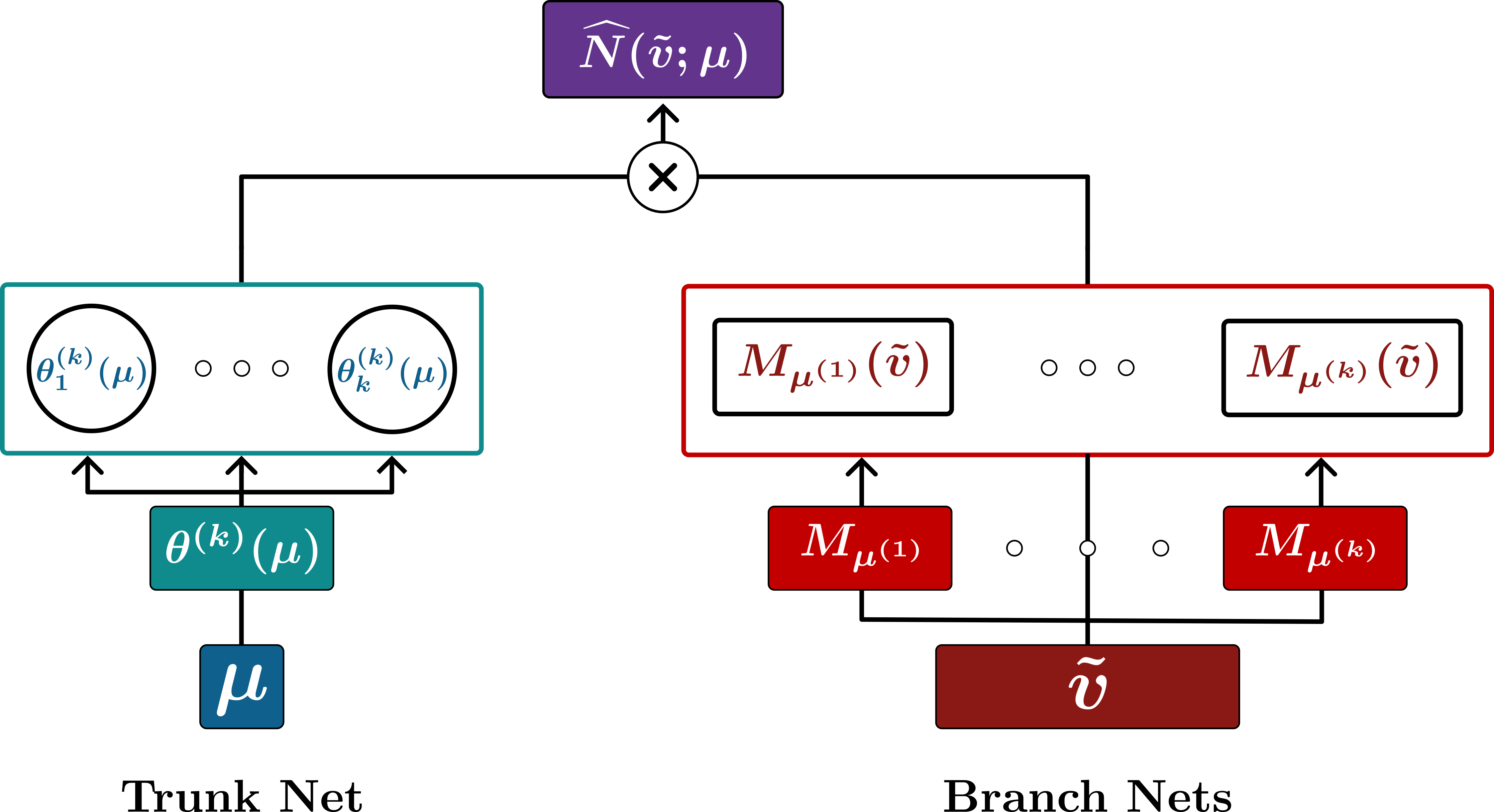}
    \caption{DeepONet schematic of NEIM.}
    \label{fig:deeponet-diagram}
\end{figure}

\section{Numerical examples}
\label{sec:numerics}
We now present some numerical examples to demonstrate the performance of NEIM\footnote[1]{Code for all examples is available at \url{https://github.com/maxhirsch/NEIM}.}. We use the open source library MLniCS \cite{mlnics}, based on RBniCS \cite{RozzaBallarinScandurraPichi2024} and FEniCS \cite{LoggAutomatedSolutionDifferential2012}, for our physics-informed neural network example. For each experiment, we use PyTorch's learning rate scheduler and the Adam optimizer for training the NEIM networks. The other parameters for the training and NEIM network architecture are listed in Table \ref{tab:architecture}.

\begin{table}[H]
    \centering
    \begin{tabular}{|c|c|c|c|c|}
    \hline
       Example & Hidden Layers & Neurons in Hidden Layer & Epochs\\
    \hline\hline
        Solution Independent & 1 & 1 & 20000 \\
        \hline
        Solution Dependent & 1 & 50 & 10000 \\
        \hline
        PINN & 1 & 10 & 10000 \\
        \hline
        Liquid Crystal $Q$, $r$ & 1 & 30 & 10000 \\
        \hline
        Liquid Crystal $\theta$ & 2 & 30 & 30000 \\
        \hline
    \end{tabular}
    \caption{Architectures exploited for the numerical examples showing the number of hidden layers for each network in the NEIM expansion, the number of neurons in each hidden layer, and the number of epochs for each network in the expansion.}
    \label{tab:architecture}
\end{table}
For the liquid crystal experiment, our two NEIM approximations used the same training parameters, which we list in the ``Liquid Crystal $Q$, $r$'' row. The ``Liquid Crystal $\theta$'' row in the table lists the parameters for training the neural network which outputs the prediction of $\theta(\mu)$. \reviewerA{We tested other interpolation methods for $\theta(\mu)$ in the other experiments depending on the complexity of the benchmark. While radial basis function interpolation could be used, it did not offer much improvement in our tests, so we used piecewise cubic and piecewise} linear interpolation for the prediction of $\theta(\mu)$ \reviewerA{instead}.
In what follows, the ``number of modes'' used in DEIM will refer to the number of components of the nonlinearity which DEIM selects to interpolate. 
In NEIM, the ``number of modes'' will simply refer to the number of terms in the NEIM expansion.

\subsection{Solution Independent Function Approximation}
\label{subsection:solution_independent_function_approximation}
In this first example, we approximate a parameterized function that can be thought of as the forcing term in a parameterized PDE discretization. Let us consider $s:\Omega\times\mathcal{P}\to\mathbb{R}$ defined by $s(x;\mu) = (1-x)\cos(3\pi\mu(x+1))e^{-(1+x)\mu},$ where $\Omega = [-1, 1]$ and $\mu\in\mathcal{P} = [1,\pi]$, \reviewerA{and the Poisson problem defined by
\begin{equation}
\label{eq:poisson}
    \begin{cases}
        \Delta v(x) = s(x; \mu),&x\in\Omega,\\
        v(-1) = v(1) = 0.
    \end{cases}
\end{equation}}
Now let $\boldsymbol{x} = [x_1,\dots,x_n]^\top\in\mathbb{R}^n$ with $x_i$ equally spaced points in $\Omega$ for $i=1,\dots,n$, with $n=100$, and  define the forcing $f:\mathcal{P}\to\mathbb{R}^n$ by $f(\mu) = [s(x_1;\mu),\dots,s(x_n;\mu)]^\top\in\mathbb{R}^n$ for $\mu\in\mathcal{P}$. The goal is to solve the equation 
\[
    h^{-2}Av_{2:n-1} = f(\mu)_{2:n-1},
\]
coming from a finite difference approximation of \reviewerA{\eqref{eq:poisson},} where $v_1 = v_n = 0$, $h^{-2} = 30$, and
\[
    A = \begin{bmatrix}
        2 & -1 &&&\\
        -1 & 2 & -1 & &\\
        & -1 & 2 & \ddots & \\
        & & \ddots & \ddots & -1\\
        & & & -1 & 2
    \end{bmatrix} \in \R^{(n-2)\times (n-2)}.
\]
Here, the subscript in $v_{2:n-1}$ and $f(\mu)_{2:n-1}$ denote the vectors in $\R^{n-2}$ containing the components $2,3,\dots,n-1$ of $v$ of $f(\mu)$, respectively. 

Solving this equation we obtain the snapshots $v(\mu_j)$ for $\{\mu_j\}_{j=1}^m \in\mathcal{P}$, with $m=51$ snapshots, and perform a POD to obtain the matrix $U_r$ of the first $r$ left singular vectors.  The parameters $\{\mu_j\}_{j=1}^m$ are selected as equally spaced points in $\mathcal{P} = [1,\pi]$. Here, we aim at comparing the performance of NEIM and DEIM approaches to approximate the nonlinearity $U_r^\top N(U_r\tilde v; \mu) \equiv U_r^\top f(\mu)$.

By looking at the POD singular values to obtain the matrix $U_r$ in Figure \ref{fig:singular-values-experiment-1}, we observe that for $r\geq30$ they are numerically $0$. Since these are not meaningful for the expansion, we choose $r = 30$.
\begin{figure}[ht]
  \captionsetup{width=0.45\textwidth}
\begin{minipage}[c]{0.49\textwidth}
    \centering
    \includegraphics[width=\textwidth]{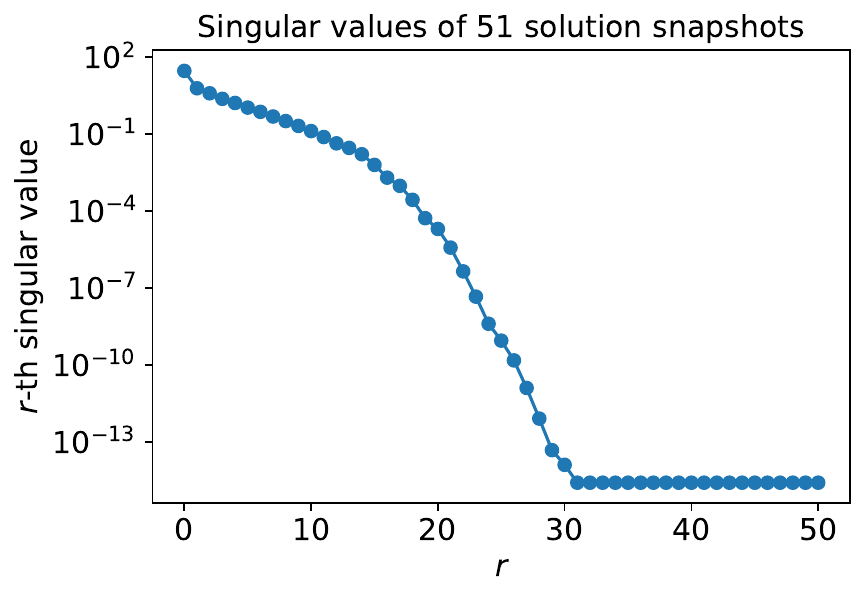}
    \caption{Singular values of snapshot matrix in experiment \ref{subsection:solution_independent_function_approximation}.}
    \label{fig:singular-values-experiment-1}
\end{minipage}\hfill
\begin{minipage}[c]{0.49\textwidth}
  \centering
  \includegraphics[width=\textwidth]{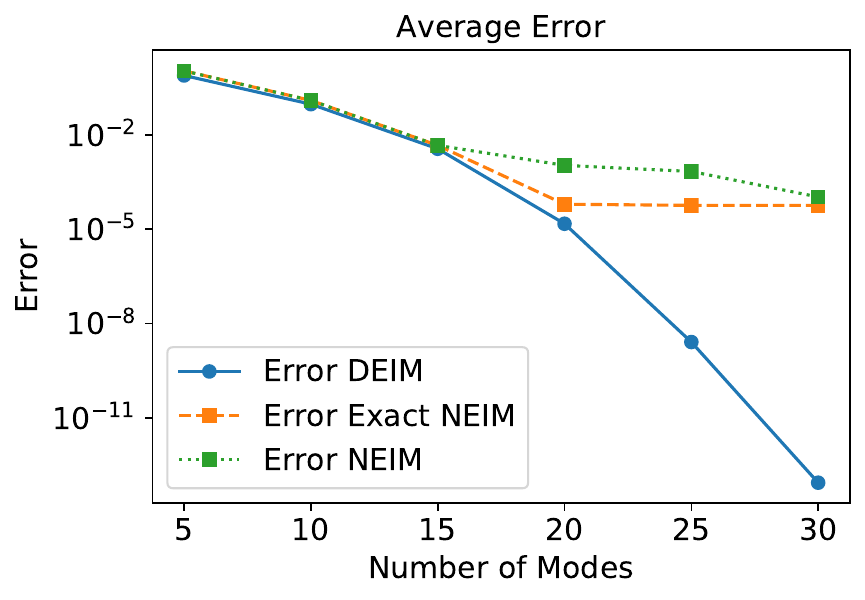}
  \caption{NEIM and DEIM average absolute error by number of modes in \ref{subsection:solution_independent_function_approximation}.}
  \label{fig:average-error-experiment-1}
\end{minipage}
\end{figure}
We now compute \reviewerCommon{a} NEIM approximation, where we take all weights equal to $1$, and \reviewerCommon{we} compare \reviewerCommon{it} with \reviewerCommon{a} DEIM approach. Because the NEIM neural networks in this case approximate a \reviewerA{constant vector-valued function} at each iteration \reviewerA{as $f(\mu)$ is independent of $v$}, we also include a comparison to ``exact'' NEIM, that is, the result of NEIM when the neural networks are replaced with the exact constant vectors \reviewerA{$M_{\mu^{(i)}}(\widetilde{v})\equiv c_i$ for some $c_i\in\R^r$}. This way, we can analyze the performance of NEIM in relation to the accuracy of the different components in the expansion.

In Figure \ref{fig:increasing-modes-experiment-1}, we see how the error in NEIM and DEIM develop as the number of modes used increases. The number of modes in DEIM is the number of singular vectors retained in the DEIM POD of the nonlinearity, while the number of modes in NEIM is the number of neural networks in the NEIM expansion. The error for all methods decreases with respect to the number of modes used. For larger numbers of modes, reaching machine precision for the singular values, DEIM outperforms NEIM, but when considering only fewer modes, NEIM  performance are remarkably great over the whole parameter space. Indeed, for $r < 20$ the neural empirical interpolation method shows small error at the parameters that have been selected, outperforming DEIM for many parameter values. \reviewerA{We remark that NEIM results in small error at the selected parameters $\mu^{(j)}$ primarily because for those parameters we obtain a best fit to $v\mapsto U^\top \operatorname{N}(v; \mu^{(j)})$, while the nonlinearities $v\mapsto U^\top \operatorname{N}(v; \mu_i)$ for $\mu_i\ne\mu^{(j)}$ may not be close to the span of the trained networks in the NEIM expansion.} DEIM seems to have roughly the same error across the parameter space for each number of modes used, while NEIM shows error of different order of magnitude. As expected, the error for the exact version of NEIM, exploiting the exact constant vectors as modes, is smaller than that of the normal NEIM when using a large number of modes. Despite this, we remark that the accuracy of NEIM and NEIM ``exact" are equivalent for $r < 20$, confirming the consistency of the approach.
\begin{figure}[ht]
\begin{tabular}{cc}
  \includegraphics[width=.45\textwidth]{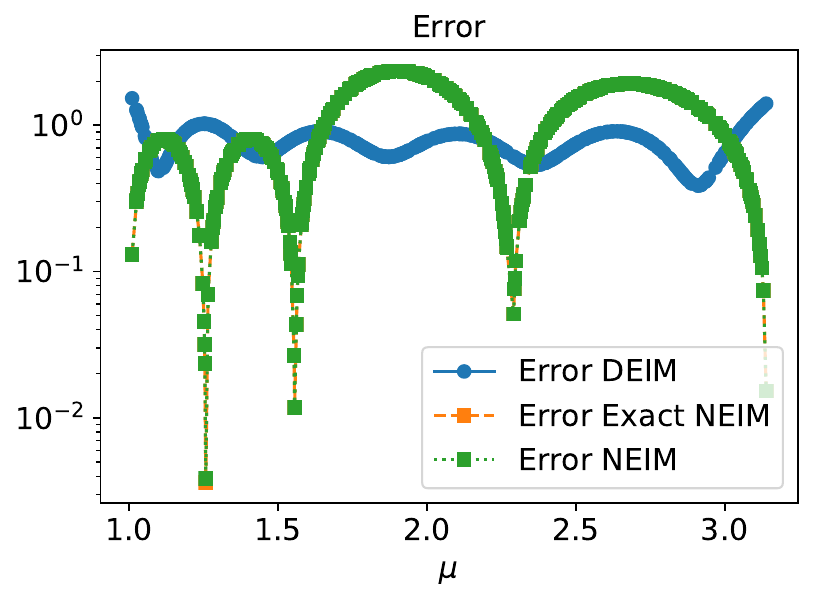} &   \includegraphics[width=.45\textwidth]{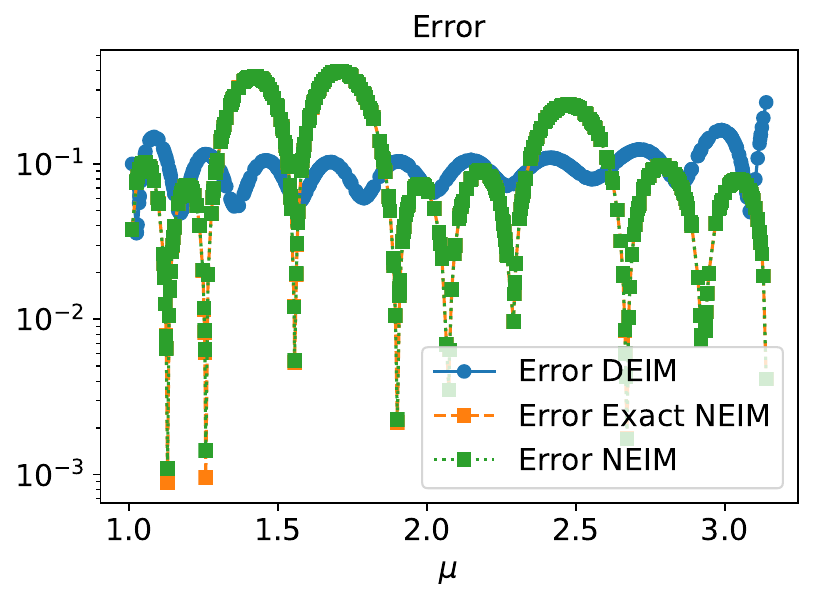} \\
(a) 5 modes & (b) 10 modes \\[6pt]
 \includegraphics[width=.45\textwidth]{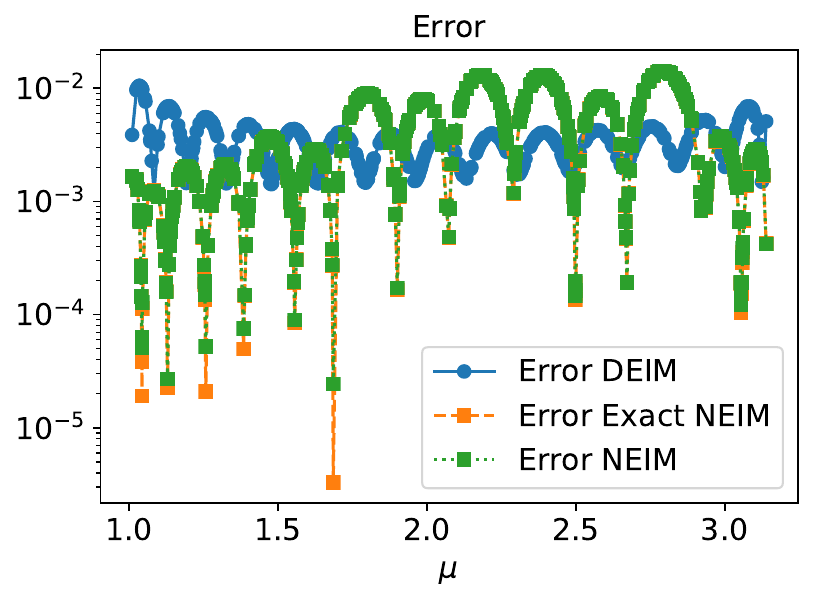} &   \includegraphics[width=.45\textwidth]{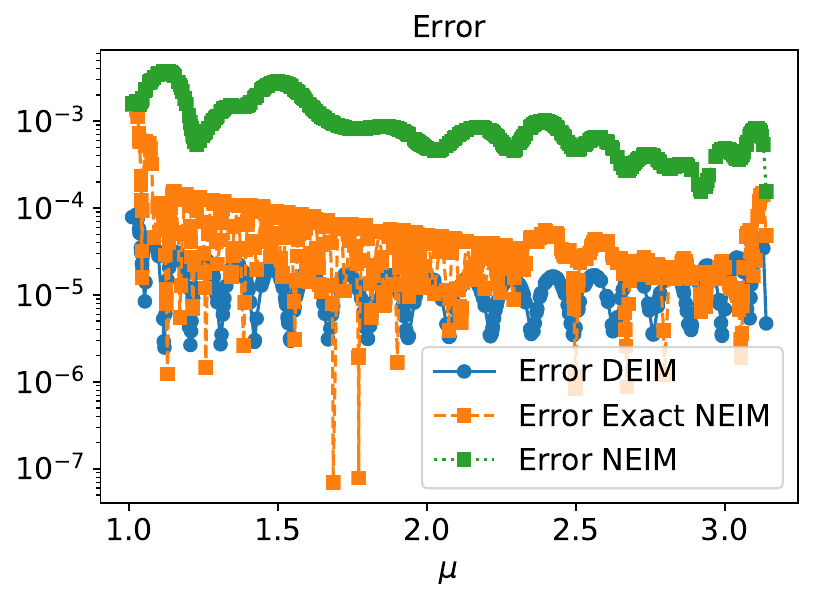} \\
(c) 15 modes & (d) 20 modes \\[6pt]
\includegraphics[width=.45\textwidth]{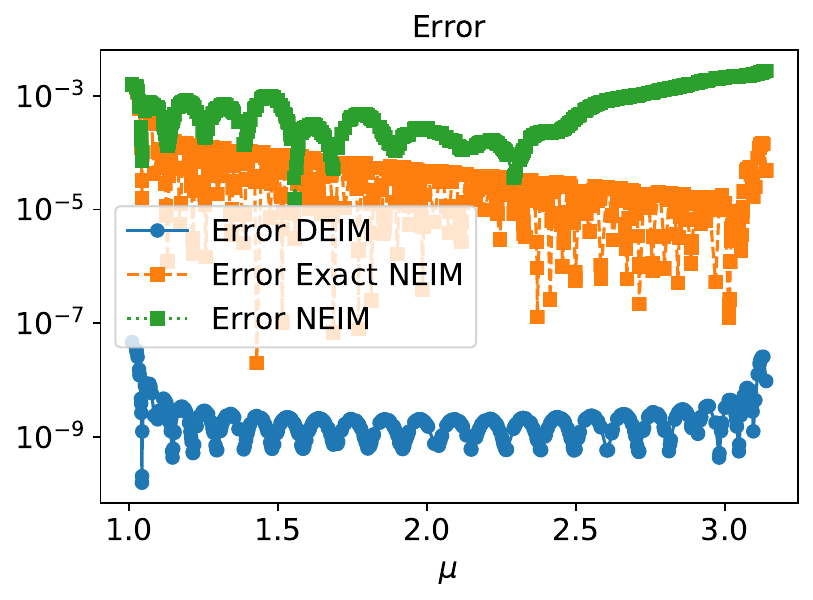} & \includegraphics[width=.45\textwidth]{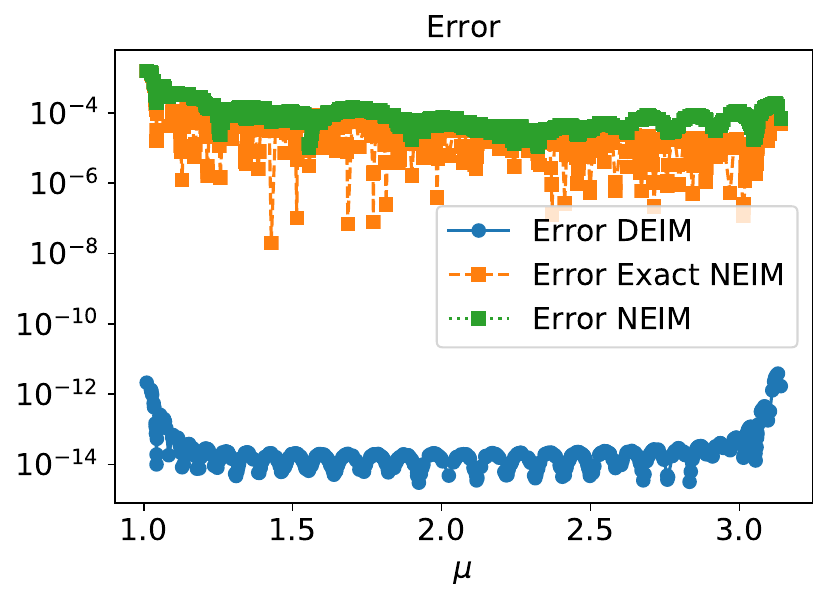}\\
(e) 25 modes & (f) 30 modes
\end{tabular}
\caption{NEIM and DEIM absolute error increasing number of modes in \ref{subsection:solution_independent_function_approximation}.}
\label{fig:increasing-modes-experiment-1}
\end{figure}
In Figure \ref{fig:average-error-experiment-1}, we see how the nonlinearity absolute error averaged over parameters changes with respect to the number of modes used in the NEIM and DEIM algorithms. This error is computed as
\[
    \frac{1}{m_{\text{test}}}\sum_{j=1}^{m_{\text{test}}} \|\widehat{f}(\mu) -U_r^\top f(\mu)\|_2,
\]
where $m_{\text{test}}=500$ is the number of parameters for which we compute the error, and $\widehat{f}$ is the DEIM or NEIM approximation of $U_r^\top f(\mu)$. We see that the error for NEIM and DEIM decay similarly for small numbers of modes, but the error for NEIM eventually show a plateau, while the error for DEIM continues to decrease.
% \begin{figure}[ht]
%     \centering
%     \includegraphics[width=.5\textwidth]{Experiment1/error_by_modes_experiment_1.pdf}
%     \caption{NEIM and DEIM average absolute error by number of modes in \ref{subsection:solution_independent_function_approximation}.}
%     \label{fig:average-error-experiment-1}
% \end{figure}

\subsection{Solution Dependent Function Approximation}
\label{subsection:solution_dependent_function_approximation}
We now approximate the parameterized nonlinear function $s:\Omega\times\R\times\mathcal{P}\to\R$ defined by $
    s(x,v;\mu) = (1 - \lvert x\rvert)e^{-(1+x)v\mu}$, where $\Omega = [-1,1]$, $\mu\in\mathcal{P}=[1,\pi]$, \reviewerA{appearing as a nonlinear term in
    \begin{equation}
        \begin{cases}
            \Delta v(x) = s(x,v(x);\mu),\quad x\in\Omega,\\
            v(-1) = v(1) = 0.
        \end{cases}
    \end{equation}
}
\reviewerCommon{Now let $x_i$ be} equally spaced points in $\Omega$ for $i=1,\dots,n$, with $n=100$\reviewerCommon{, and in an abuse of notation, let} $v(\mu)\in\R^n$ be the solution of the \reviewerCommon{discretized} system,
\begin{equation}
    \begin{bmatrix}\vspace{-0.1cm}
        0 & \dots & 0\\
        \vdots & h^{-2}A & \vdots\\
        0 & \dots & 0
    \end{bmatrix}v - f(v; \mu) = 0, \qquad \text{with} \qquad v_1 = v_n = 0,
\end{equation}
where $h^{-2} = 30$ and $f(v;\mu)$ is the nonlinear term in the system defined by $f(v; \mu) = \begin{bmatrix} s(x_1, v_1; \mu) & \dots & s(x_n, v_n; \mu)\end{bmatrix}^\top$.
As before, we use $m=51$ equally spaced parameters in $\mathcal P = [1,\pi]$, and perform a POD on the matrix of solution snapshots to obtain a POD basis $U_r$. The goal now is to approximate $U_r^\top f(U_r\tilde v; \mu)$ for $\mu\in\mathcal{P}$.
Based on the plot of singular values in Figure \ref{fig:singular-values-experiment-2}, we truncate our expansion at $r=20$ modes.
\begin{figure}[ht]
  \captionsetup{width=0.45\textwidth}
  \begin{minipage}[c]{0.49\textwidth}
    \centering
    \includegraphics[width=\textwidth]{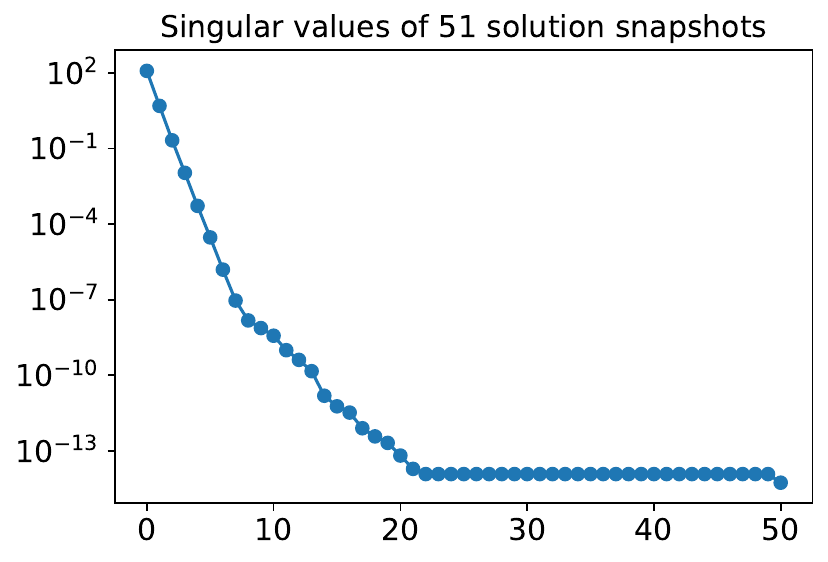}
    \caption{Singular values of snapshot matrix for experiment \ref{subsection:solution_dependent_function_approximation}.}
    \label{fig:singular-values-experiment-2}
\end{minipage}\hfill
\begin{minipage}[c]{0.49\textwidth}
  \centering
  \includegraphics[width=\textwidth]{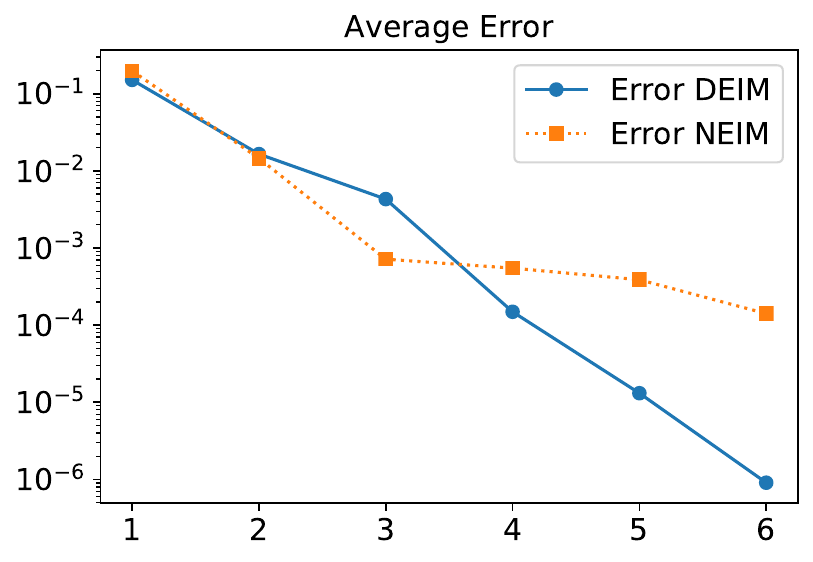}
  \caption{NEIM and DEIM average absolute error by number of modes in \ref{subsection:solution_dependent_function_approximation}.}
  \label{fig:average-error-experiment-2}
\end{minipage}\hfill
\end{figure}
For this test case, we take the weights in the NEIM algorithm to be
    $w_e(\mu_i; \mu_j) = \delta_{ij},\quad w_t^{(k)}(\mu_i) = 1\quad \forall i,j,k$,
and compare the accuracy results of NEIM and DEIM approaches.

In Figure \ref{fig:increasing-modes-experiment-2}, we see how the error in NEIM and DEIM develop as the number of modes used increases. We see how the greedy procedure led NEIM to initially select parameters close to the boundary of parameter space, and the errors for NEIM and DEIM both decrease as the number of modes increases, as expected. Also in this case, the error for NEIM is lowest for selected parameters, while both approaches attempt to control the error across the parameter space. While the DEIM error continues to decrease, eventually the NEIM error ceases to decrease significantly, but not before having reached an accuracy of order $10^{-4}$. Being a completely data-driven procedure such approximation property is remarkable, especially considering we are not exploiting any physics.
\begin{figure}[ht]
\begin{tabular}{cc}
  \includegraphics[width=0.45\textwidth]{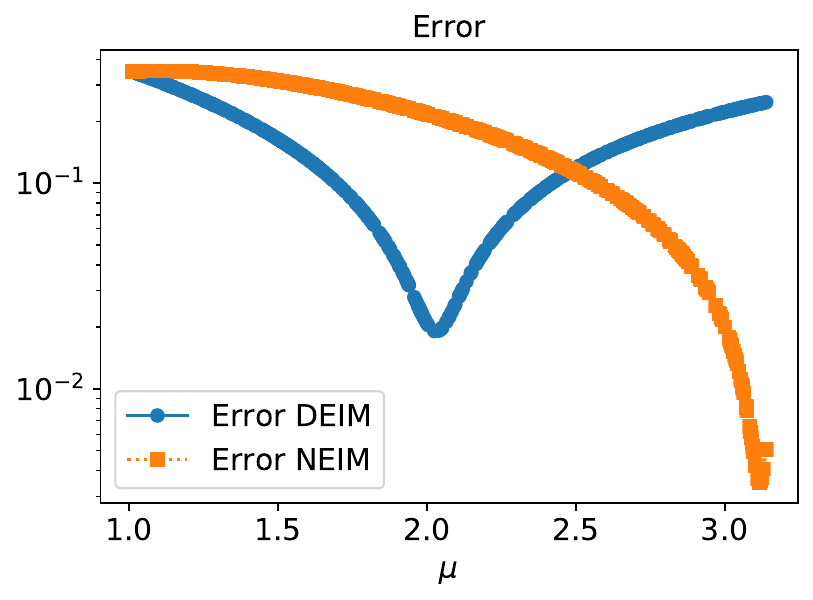} &   \includegraphics[width=0.45\textwidth]{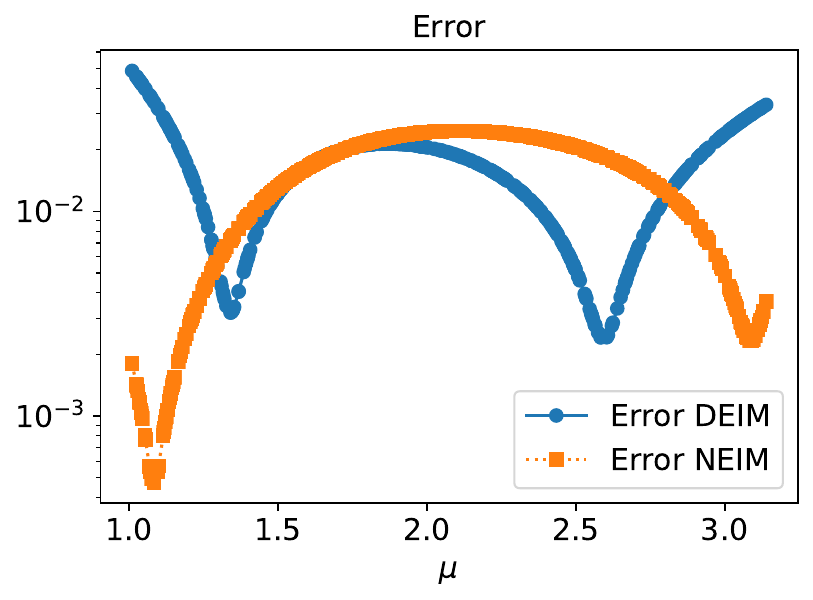} \\
(a) 1 mode & (b) 2 modes \\[6pt]
 \includegraphics[width=0.45\textwidth]{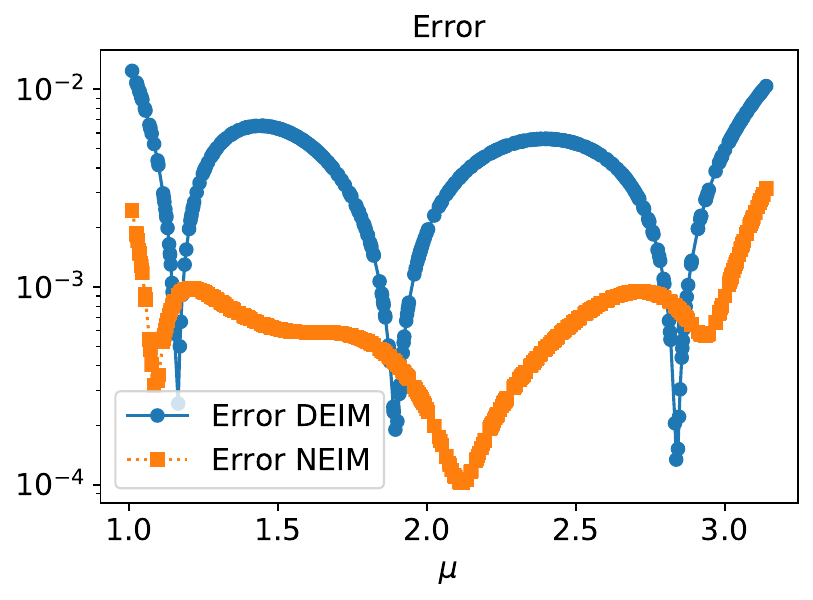} &   \includegraphics[width=0.45\textwidth]{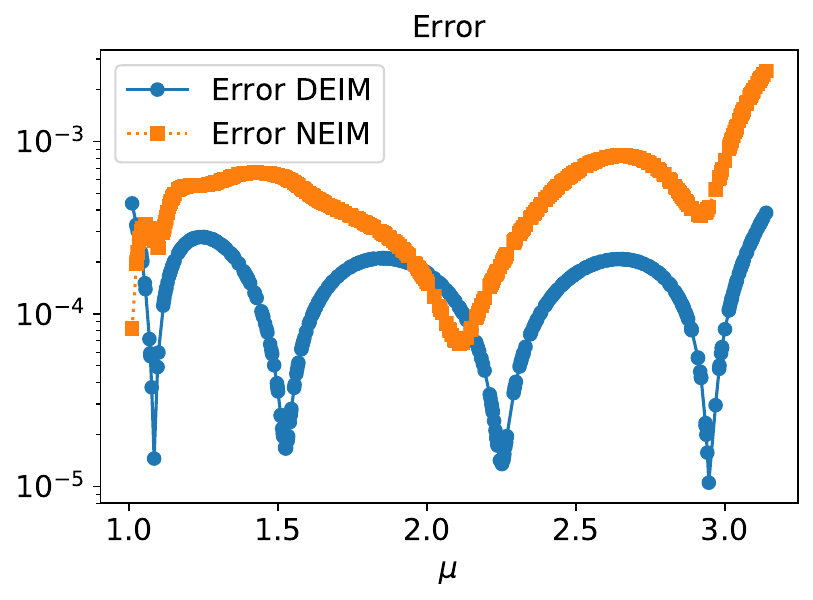} \\
(c) 3 modes & (d) 4 modes \\[6pt]
\includegraphics[width=0.45\textwidth]{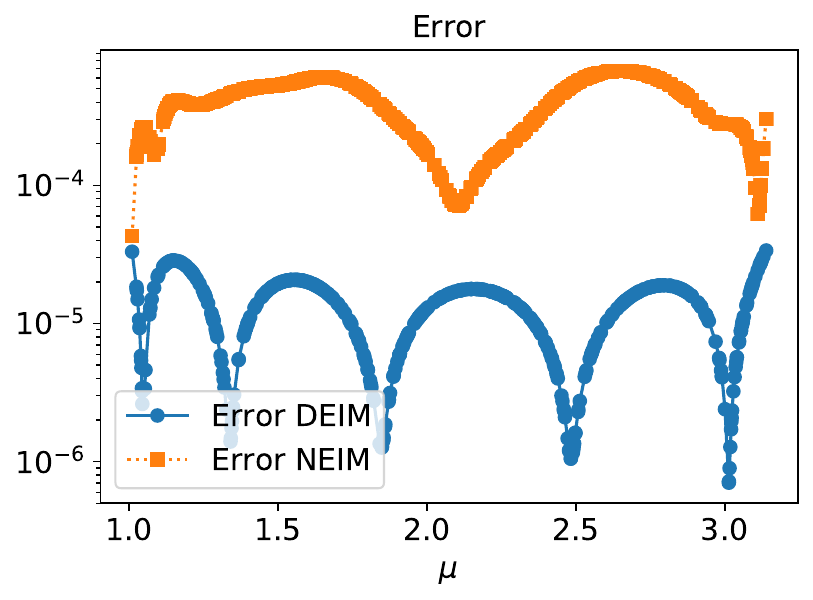} & \includegraphics[width=0.45\textwidth]{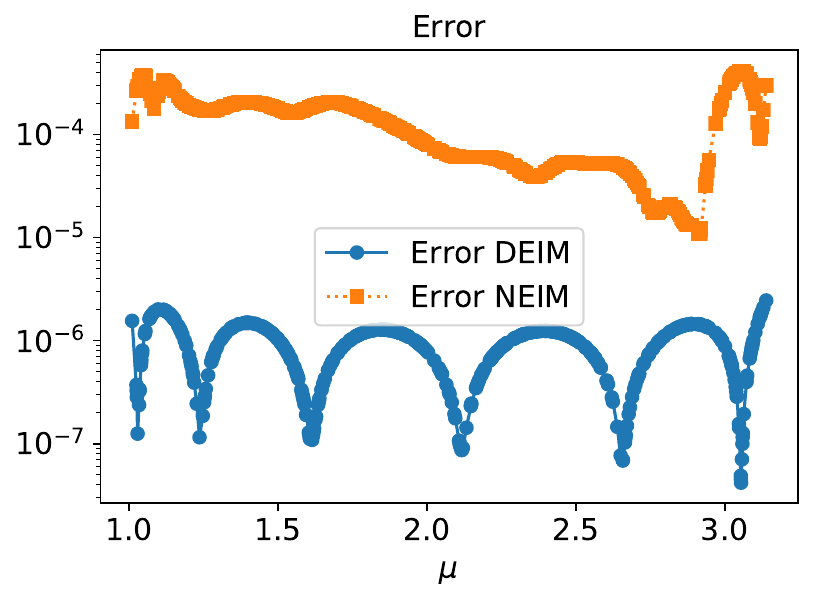}\\
(e) 5 modes & (f) 6 modes
\end{tabular}
\caption{NEIM and DEIM absolute errors by parameter for increasing number of modes in experiment \ref{subsection:solution_dependent_function_approximation}.}
\label{fig:increasing-modes-experiment-2}
\end{figure}
As before, we show in Figure \ref{fig:average-error-experiment-2} how the absolute error averaged over parameters changes with respect to the number of modes used in the NEIM and DEIM algorithms. Here we see more clearly the error decays for both NEIM and DEIM, which exhibit monotone behaviors, even though NEIM error starts flattening out.

\reviewerB{Now in Figure \ref{fig:NEIM-DEIM-modes}, we show examples of the learned modes in NEIM and DEIM. The exact nonlinearity in (a) is $U_rU_r^\top f(U_r\tilde{v}(\mu); \mu)$ for $\mu=1.72$. In (b), we see how the DEIM error is smallest at the selected interpolation points in the DEIM algorithm, as expected. Then plots (c)-(f) are the first six NEIM and DEIM modes for $\mu=1.72$ and $\mu=1.00$. In NEIM, these are each of the first six terms in the affine expansion, and in DEIM, these are the residuals approximated by each additional DEIM mode used. The NEIM and DEIM modes in (c) and (d) correspond to $\mu=1.72$, which was not one of the selected parameters in the NEIM algorithm. We see that both NEIM and DEIM have smooth approximations in the spatial domain, with primarily the first three modes in each being nonzero. Visually, we see that these modes indeed seem to add to the nonlinearity in (a). When $\mu=1$, which was the second parameter selected by the NEIM algorithm, only the first and second NEIM modes are significantly different from zero, which we expect since NEIM should almost exactly approximate the $j$th selected parameter with its first $j$ modes. It was also observed that the only significantly nonzero term in the NEIM approximation corresponding to the first selected parameter in the algorithm was the first mode.}

\reviewerA{Lastly, we compare our NEIM approximation of the nonlinearity with a simple feedforward neural network approximation of $\mu\mapsto U_r^\top f(U_r\tilde{v}(\mu); \mu)$ in Figure \ref{fig:NEIM-PODNN-comparison}. We tested feedforward networks containing more parameters than NEIM and also fewer parameters than NEIM. In both cases, NEIM performed better than or comparably to the feedforward networks tested. Thus, in this example, we see that the greedy training algorithm performs better than training everything at once with a single network.}

\begin{figure}[ht]
\begin{tabular}{cc}
  \includegraphics[width=0.45\textwidth]{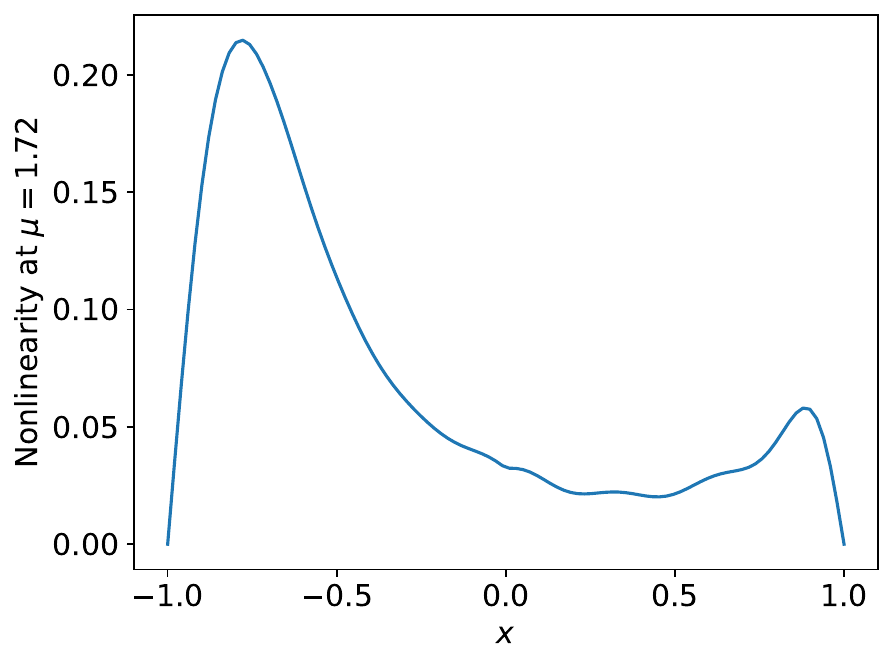} &   \includegraphics[width=0.45\textwidth]{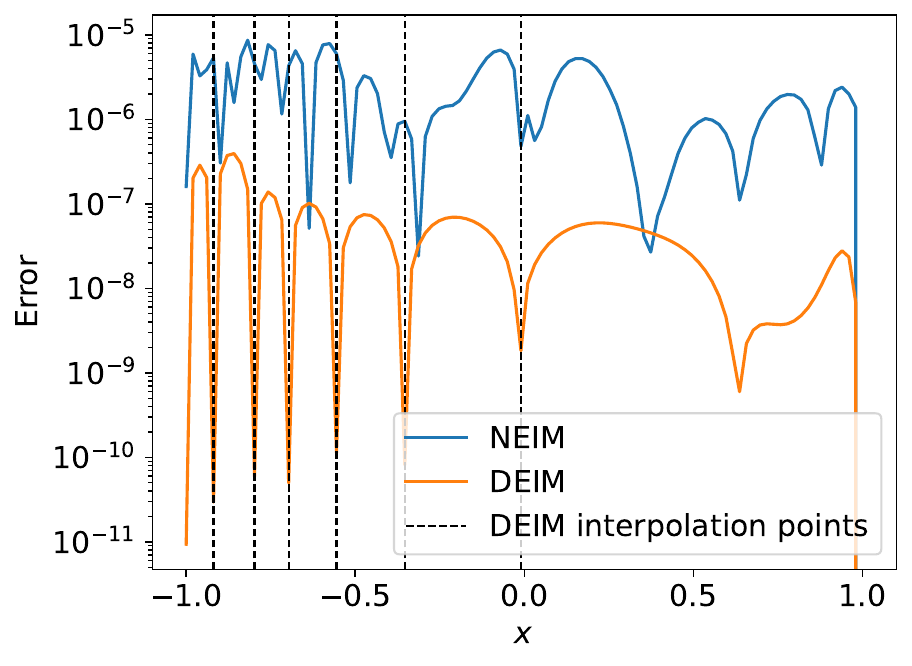} \\
\reviewerB{(a) Exact Nonlinearity $\mu=1.72$} & \reviewerB{(b) NEIM/DEIM Error $\mu=1.72$}\\[6pt]
 \includegraphics[width=0.45\textwidth]{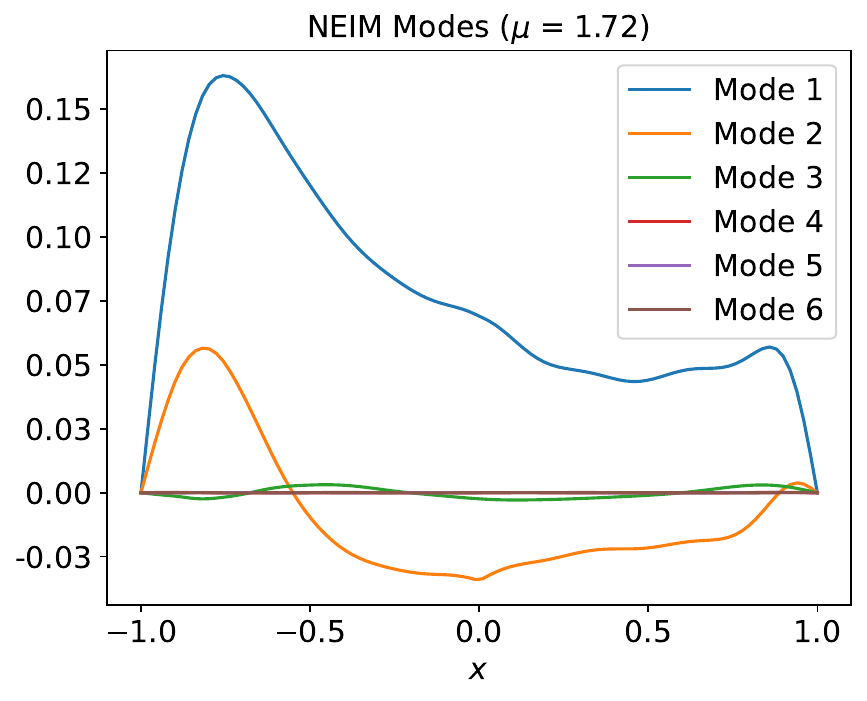} &   \includegraphics[width=0.45\textwidth]{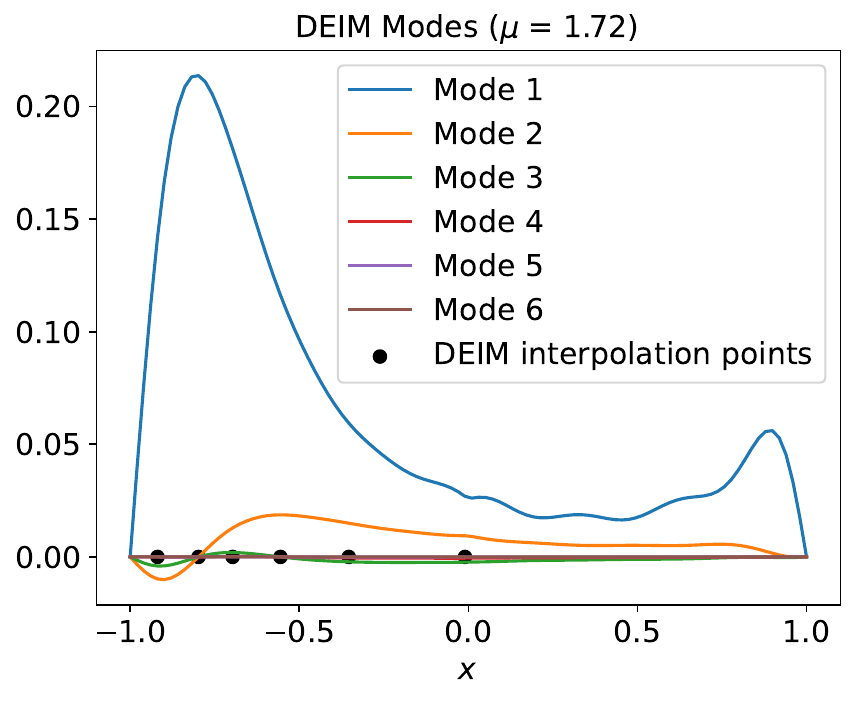} \\
\reviewerB{(c) NEIM Modes $\mu=1.72$} & \reviewerB{(d) DEIM Modes $\mu=1.72$} \\[6pt]
\includegraphics[width=0.45\textwidth]{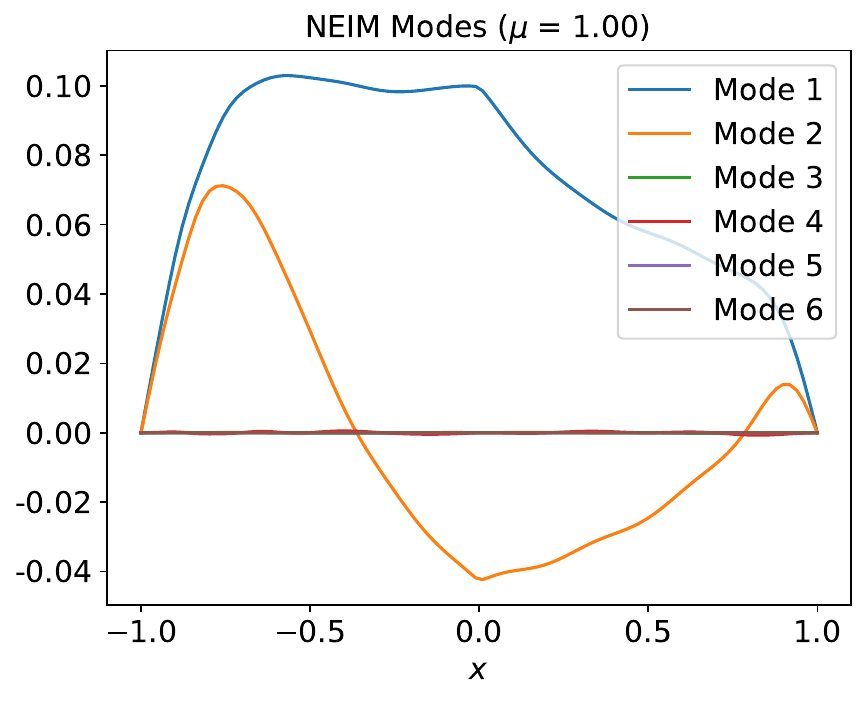} & \includegraphics[width=0.45\textwidth]{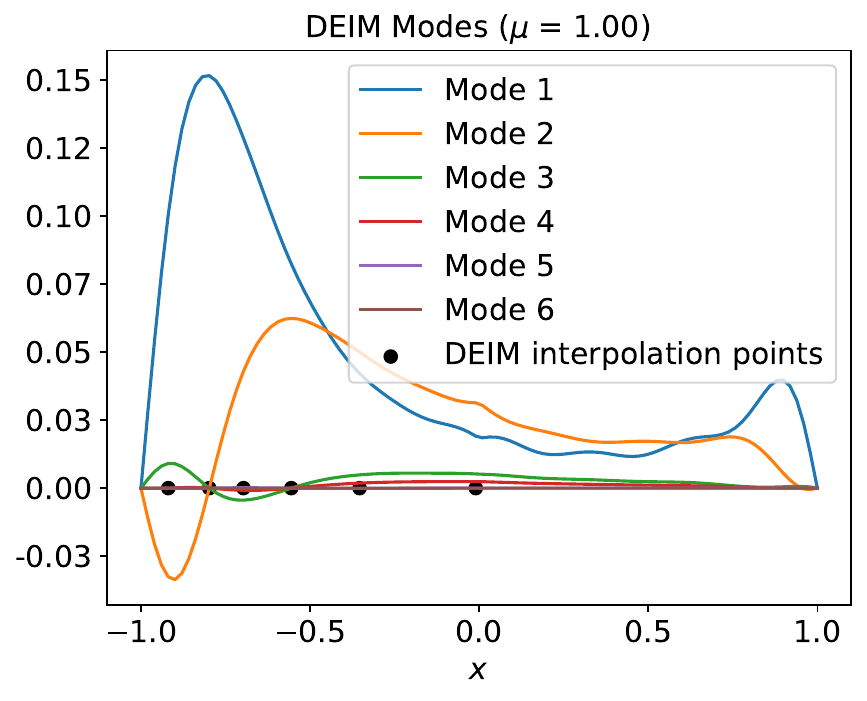}\\
\reviewerB{(e) NEIM Modes $\mu=1.00$} & \reviewerB{(f) DEIM Modes $\mu=1.00$}
\end{tabular}
\caption{\reviewerB{NEIM and DEIM modes in experiment \ref{subsection:solution_dependent_function_approximation}.}}
\label{fig:NEIM-DEIM-modes}
\end{figure}

\begin{figure}[ht]
    \centering
    \includegraphics[width=0.8\linewidth]{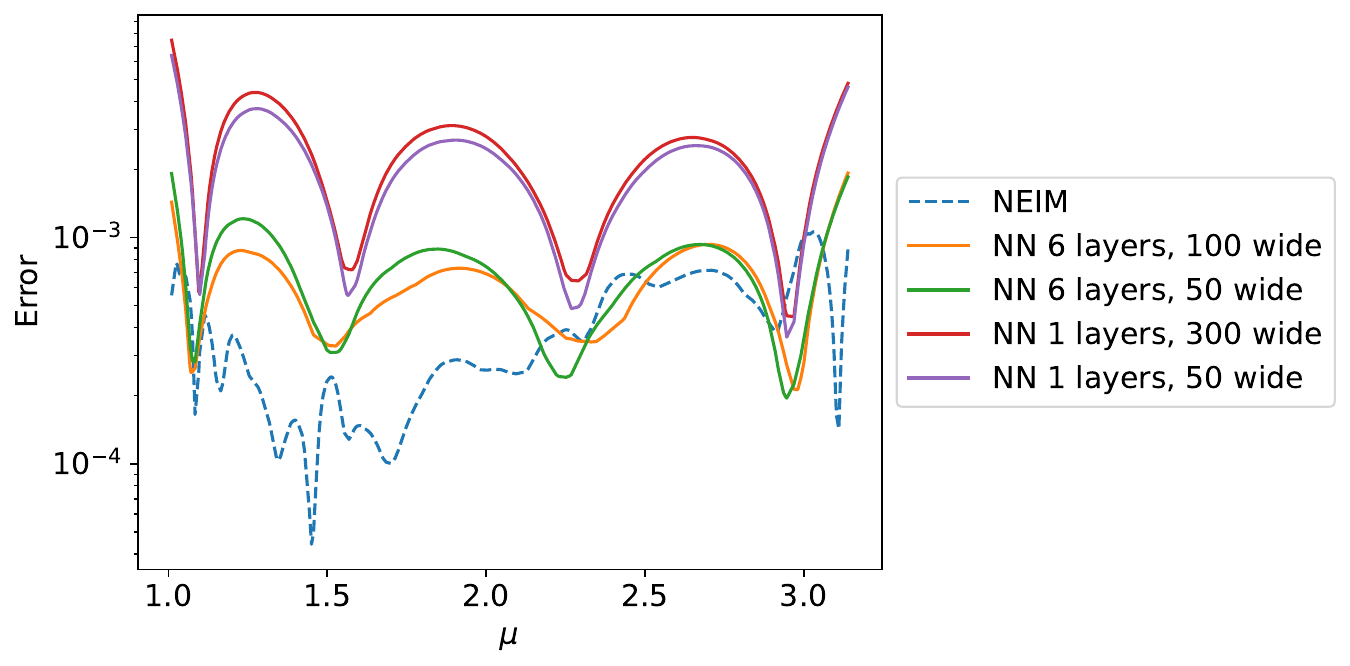}
    \caption{\reviewerA{Absolute error of NEIM approximation of nonlinearity versus basic neural network approximation for various architectures in experiment \ref{subsection:solution_dependent_function_approximation}.}}
    \label{fig:NEIM-PODNN-comparison}
\end{figure}

% \begin{figure}[ht]
%     \centering
%     \includegraphics[width=.45\textwidth]{Experiment2/error_by_modes_experiment_2.pdf}
%     \caption{NEIM and DEIM average absolute error by number of modes in experiment \ref{subsection:solution_dependent_function_approximation}.}
%     \label{fig:average-error-experiment-2}
% \end{figure}

\subsection{Nonlinear Elliptic Physics-Informed Neural Network}
\label{subsection:PINN}
We now apply NEIM in the context of physics-informed neural networks (PINNs) for reduced order models. The goal is to combine a physics-based learning within a data-driven reduced context. More specifically, one aims at learning the parameter-to-reduced vector mapping via neural network architectures that minimize \reviewerCommon{an unsupervised loss term based on the reduced residual projected onto the linear POD subspace.} \reviewerA{The setup for the problem is shown in Figure \ref{fig:PINN-NEIM-architecture}.} %joint loss given by a supervised term \reviewerCommon{which accounts} for \reviewerCommon{the} projection onto a linear POD subspace, and an unsupervised term based on the reduced residual projected onto the same base. 
\begin{figure}[ht]
    \centering
    \includegraphics[width=0.9\linewidth]{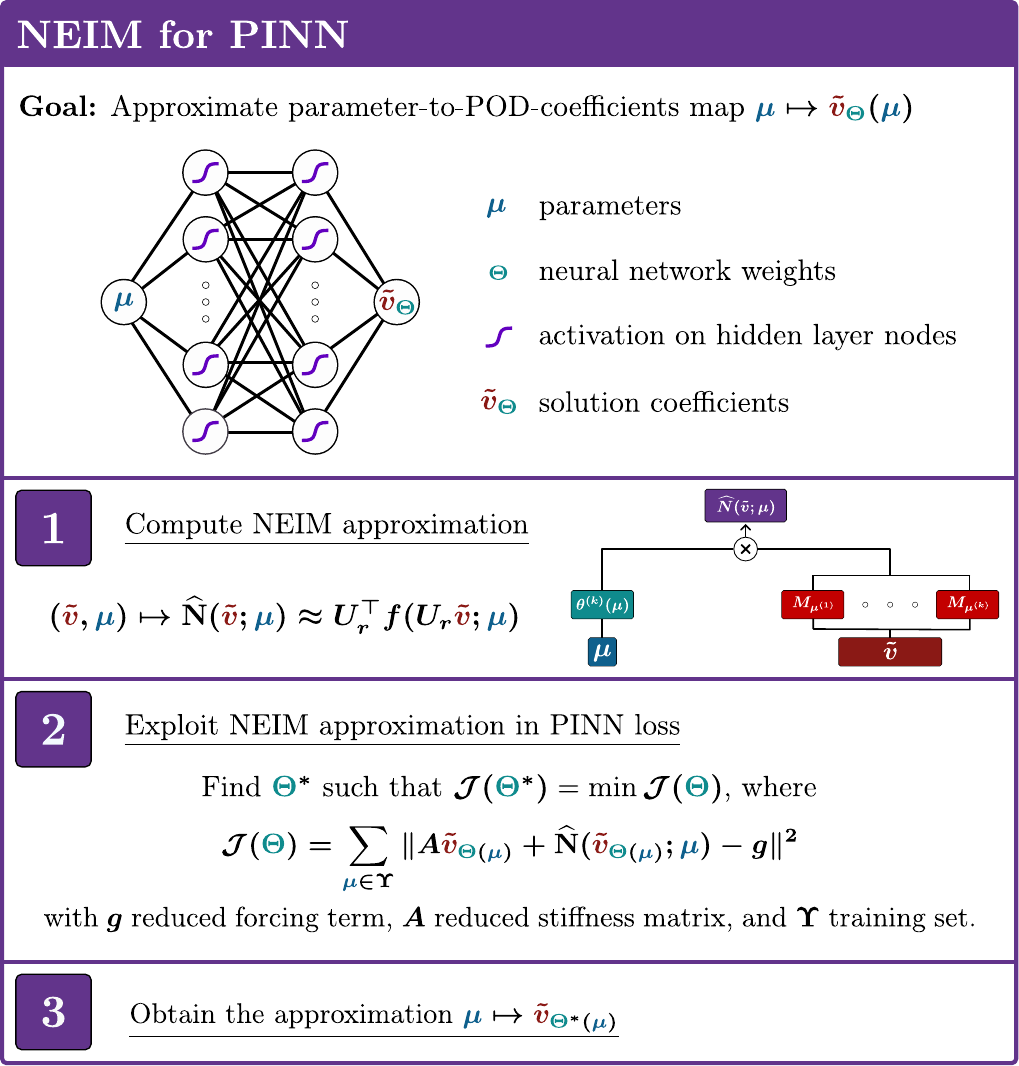}
    \caption{\reviewerA{Setting for NEIM with physics-informed neural networks in experiment \ref{subsection:PINN}.}}
    \label{fig:PINN-NEIM-architecture}
\end{figure}
For further details we \reviewerA{refer} the reader to \cite{HesthavenNonintrusiveReducedOrder2018, ChenPhysicsinformedMachineLearning2021, PichiArtificialNeuralNetwork2023}. In order to embed the reduced physics in the loss for nonlinear problems, it is important that the approximation of the nonlinearity is compatible with automatic differentiation.
When it comes to more complex nonlinearities, exploiting tensor-assembly for polynomial nonlinearities or empirical interpolation approaches can be challenging. In fact, DEIM will require \reviewerCommon{assembling} nonlinearities in the online phase of computation, which, if done in standard finite element packages like FEniCS \cite{LoggAutomatedSolutionDifferential2012}, will not be differentiable in deep learning libraries like PyTorch \cite{PyTorch}. On the \reviewerCommon{other hand}, NEIM is automatically compatible with deep learning libraries because our approach is itself a neural network with automatic differentiation included by definition.

Using physics-informed neural networks for model order reduction with NEIM also leads to a completely data-driven and interpretable method for reduced order models. DEIM, on the other hand, still requires the finite element high fidelity operators, and thus prevents its usage in combination with black-box solvers.

Let us consider a nonlinear elliptic problem in a two-dimensional spatial domain $\Omega=(0,1)^2$. The strong formulation of the problem is given by: for a given parameter $\boldsymbol{\mu}= (\mu_1, \mu_2)\in\mathcal{P}=[0.01,10.0]^2$, find $v(\boldsymbol{\mu})$ such that
\[ 
    -\Delta v(\boldsymbol{\mu})+\frac{\mu_1}{\mu_2}(\exp^{\mu_2v(\boldsymbol{\mu})}-1)=100\sin(2\pi x)\sin(2\pi y),
\]
with $v(\boldsymbol{\mu})$ satisfying homogeneous Dirichlet boundary conditions. 

We exploit the finite element method as \reviewerCommon{our} high fidelity solver to discretize the PDE. Then, we perform a POD on $m=100$ equispaced points in parameter space to obtain $U_r$. \reviewerA{In this case, our goal is to approximate $U_r^\top f(U_r\tilde v; \mu)$, where $f$ is given by $f(v; \mu)_i = \frac{\mu_1}{\mu_2}\left(\exp\left(\mu_2v_i\right) - 1\right)\Delta x_{i},\ i=1,\dots,n$,
and $\Delta x_i$ is the volume $\int \varphi_i\, d\boldsymbol{x}$ of the piecewise linear finite element basis test function $\varphi_i$ at node $i$. In other words, the nonlinearity we aim at approximating is a low order quadrature rule for the nonlinear integral term in the FE formulation of the problem given by $\int \frac{\mu_1}{\mu_2}(\exp(\mu_2 v)-1)\varphi_i\, d\boldsymbol{x} \approx \frac{\mu_1}{\mu_2}(\exp(\mu_2 v_i)-1)\int\varphi_i\, d\boldsymbol{x}$.} The weights we use for NEIM are given by $w_e(\mu_i; \mu_j) = \delta_{ij}$ and
\[
    w_t^{(k)}(\mu_i) = \begin{cases}
        1 &\text{if } \|\mu_i - \mu^{(k)}\|_2 \le 1.75,\\
        0 &\text{otherwise}.
    \end{cases}
\]

Now we train a physics-\reviewerCommon{informed} neural network using DEIM, which we denote by PINN-DEIM, to solve the problem. The physics-\reviewerCommon{informed} neural network takes in the parameter for the problem and outputs an estimate of the reduced order solution vector. It is trained by minimizing the sum of the squared norm of the discretized equation residual.
We similarly do this for a PINN using NEIM, which we call PINN-NEIM. Each iteration of training is slower for PINN-NEIM than for PINN-DEIM due to the need to backpropagate through the NEIM neural networks, but the loss generally decreases more per iteration for PINN-NEIM. For comparison between the two networks, however, we use 1000 training iterations and two hidden layers with 60 neurons each, for both methods.

We now compare the errors of the two methods on a test set of $m_{\text{test}} = 100$ uniformly random test points. The training and test sets are shown in Figure \ref{fig:train-test-points-experiment-3}. \reviewerB{We note that in a ROM context, one would like to have the discretization of the parameter space as sparse as possible. In general, if we refine our sample for the training set, we expect that the POD would capture more information and that the greedy step in the NEIM algorithm would select more representative points with large error, leading to a more accurate approximation. In this case, we expect that the trend in the error as the number of NEIM modes increases should remain the same.}
\begin{figure}[ht]
    \centering
    \includegraphics[width=.5\textwidth]{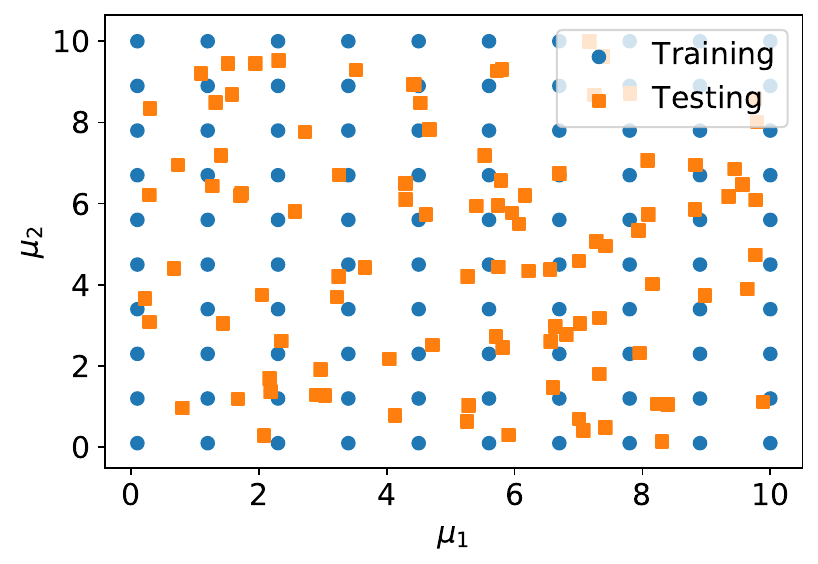}
    \caption{Training and testing parameters for experiment \ref{subsection:PINN}.}
    \label{fig:train-test-points-experiment-3}
\end{figure}

The relative error to the high fidelity solution averaged over the test set for the PINN-DEIM is about $6.12 \times 10^{-2}$, while for the PINN-NEIM, it is about $9.55 \times 10^{-3}$. Here we compute the error with the Euclidean relative error in MLniCS \cite{mlnics},
\[
    \frac{1}{m_{\text{test}}}\sum_{j=1}^{m_{\text{test}}} \frac{\|\widehat{v}(\mu_j) - v(\mu_j)\|_2}{\|v(\mu_j)\|_2}.
\]
In Figure \ref{fig:nonlinear-elliptic-solutions}, we plot the PINN-DEIM and PINN-NEIM solutions, and the error fields obtained by the difference between the approximated and the high fidelity solutions for $\mu = (5.5, 5.5)$. As we can see from the plots, the error patterns for both approaches are aligned with the physical phenomena, but the magnitude of the DEIM error is larger. Based on these results, we can see that NEIM, despite being a data-driven approach, is capable of providing meaningful recontructions of the nonlinear terms, resulting in a competitive approach with respect to DEIM in the PINN context. 
\begin{figure}[ht]
\begin{tabular}{cc}
  \includegraphics[width=0.45\textwidth]{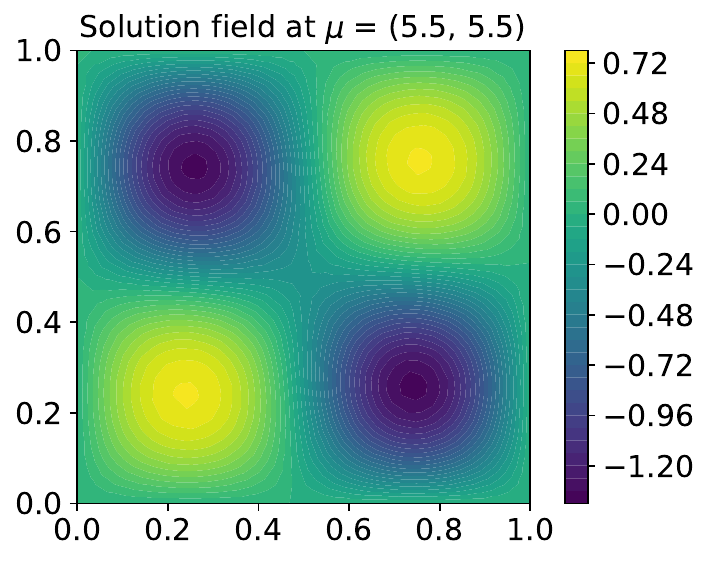} &   \includegraphics[width=0.45\textwidth]{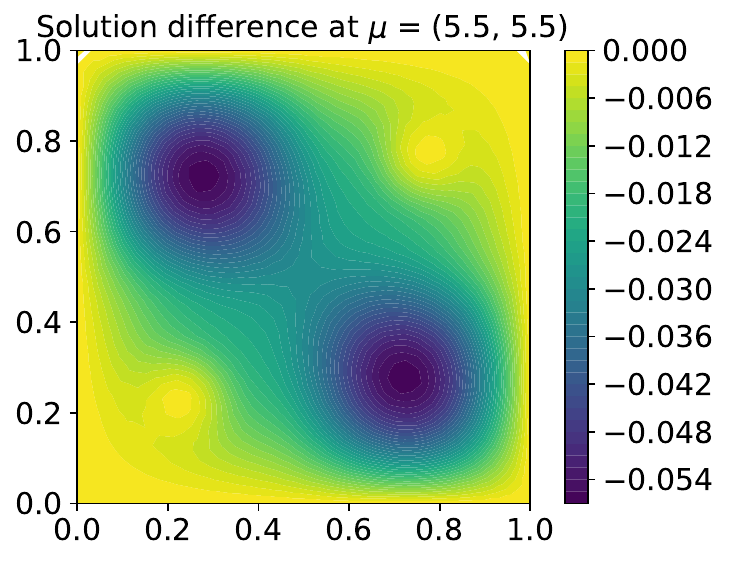} \\
(a) PINN-DEIM Solution & (b) PINN-DEIM Error \\[6pt]
 \includegraphics[width=0.45\textwidth]{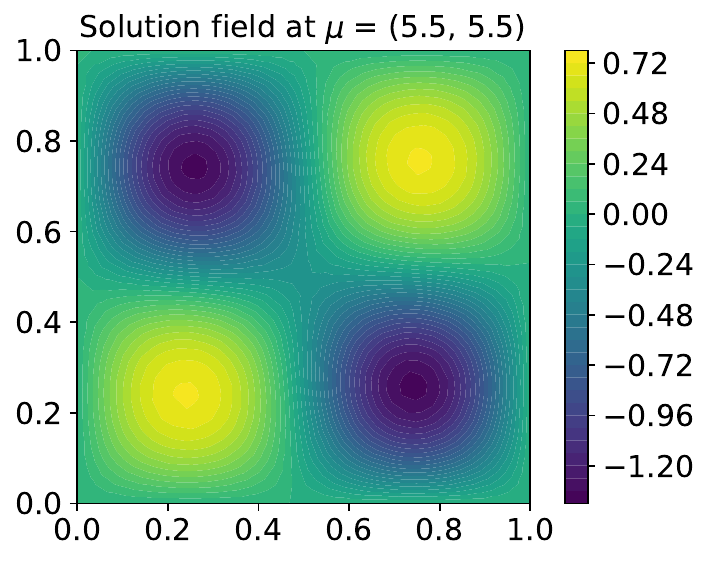} &   \includegraphics[width=0.45\textwidth]{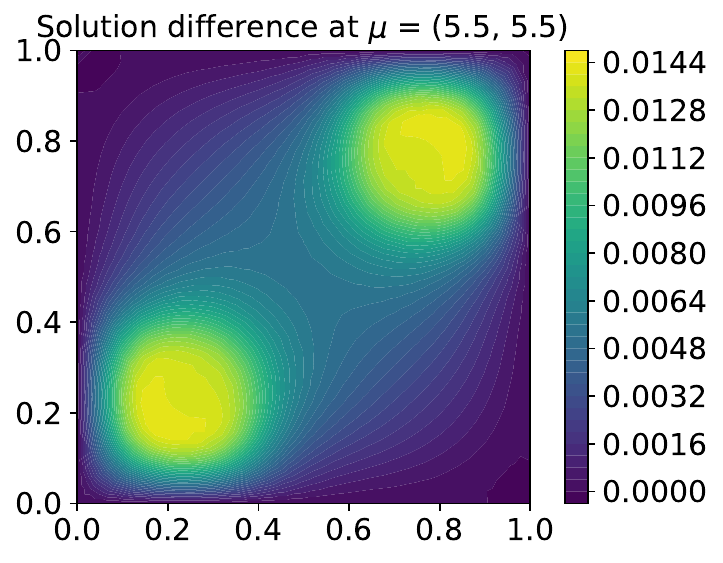} \\
(c) PINN-NEIM Solution & (d) PINN-NEIM Error \\[6pt]
\end{tabular}
\caption{PINN-DEIM and PINN-NEIM solutions and errors for $\mu=(5.5, 5.5)$ for experiment \ref{subsection:PINN}.}
\label{fig:nonlinear-elliptic-solutions}
\end{figure}

\subsection{Q-Tensor Model of Liquid Crystals Finite Elements}
\label{subsection:qtensor}
Finally, we consider a nonlinear parabolic problem in a two-dimensional spatial domain $\Omega$. The problem describes the evolution of liquid crystal molecule orientations \cite{Weber2022} by means of the coupled matrix-valued PDE:
\begin{equation}
\label{eq:qtensor}
    \begin{cases}
        Q_t &= M\left(L\Delta Q - rP(Q)\right)\\
        r_t &= \langle P(Q), Q_t\rangle_F,
    \end{cases}
\end{equation}
where the unknowns are the $2\times 2$ symmetric and trace-free Q-tensor and the scalar auxiliary variable $r$, $\langle\cdot,\cdot\rangle_F$ denotes the Frobenius inner product, and
\begin{equation}
    P(Q) = \frac{aQ - b[Q^2 - (\operatorname{tr}(Q^2)/2)I] + c\operatorname{tr}(Q^2)Q}{\sqrt{2\left(\frac{a}{2}\operatorname{tr}(Q^2) - \frac{b}{3}\operatorname{tr}(Q^3) + \frac{c}{4}\operatorname{tr}^2(Q^2) + A_0\right)}}.
\end{equation}
Because $Q$ is symmetric and trace-free \cite{Weber2022}, this problem only depends on three variables: $Q_{11}$, $Q_{12}$, and $r$. To recast the problem in the ROM setting, we consider a parameterized nonlinear term where $a\in[-0.5,0.5]$ is the varying parameter while holding the other parameters fixed at $M = 1$, $L = 0.1$, $b = 0.5$, $c = 1$, $A_0 = 500$, and we use homogeneous Dirichlet boundary conditions, and initial condition
\[
    Q_0(x, y) = dd^\top - \frac{\|d\|_2^2}{2}I,\quad d = \frac{\begin{bmatrix} (4-x^2)(4-y^2) & \sin(\pi x)\sin(\pi y)\end{bmatrix}^\top}{\sqrt{1 + [(4-x^2)(4-y^2)]^2 + \sin(\pi x)^2\sin(\pi y)^2}},
\]
where $I$ is the $2\times 2$ identity matrix. We also consider time $t\in[0,2]$ as a parameter. The parameter vector $\boldsymbol{\mu}$ is thus given by $\boldsymbol{\mu} = (t, a)$ on the parameter domain $\mathcal{P} = [0,2]\times[-0.5,0.5]$. The authors of \cite{Weber2022} show that weak solutions to \eqref{eq:qtensor} are weak solutions to the gradient flow 
\[
    Q_t = M\left(L\Delta Q - \frac{\delta\mathcal{F}_B(Q)}{\delta Q}\right),
    % \\&= M\left(L\Delta Q - \left(aQ - b\left(Q^2 - \frac12\operatorname{tr}(Q^2)I\right) + c\operatorname{tr}(Q^2)Q\right)\right),
\]
with $\frac{\delta\mathcal{F}_B(Q)}{\delta Q} =  aQ - b\left(Q^2 - \frac12\operatorname{tr}(Q^2)I\right) + c\operatorname{tr}(Q^2)Q$, so we can think of the parameter $a$ as the coefficient of the quadratic term in the quartic double-well bulk potential $\mathcal{F}_B(Q) = \frac{a}{2}\operatorname{tr}(Q^2) - \frac{b}{3}\operatorname{tr}(Q^3) + \frac{c}{4}(\operatorname{tr}(Q^2))^2$. Changing the parameter $a$ affects the depth of the two wells in the double-well potential while keeping the end behavior of the potential the same.

Due to the finite element discretization we exploit, and the complexity of the model, we need to approximate two different nonlinear terms via DEIM/NEIM to recover the solution of the problem. \reviewerA{Indeed, the time discretization for the problem is given by
\[
    \begin{cases}
        \frac{Q^{n+1} - Q^n}{\Delta t} &= M\left(\frac{L}{2}(\Delta Q^{n+1} + \Delta Q^n) - r^nP(Q^n)\right)\\
        r^{n+1} - r^n &= \langle P(Q^n), Q^{n+1}-Q^n\rangle_F,
    \end{cases}
\]
so we see that both $r^nP(Q^n)$ and $\langle P(Q^n), Q^{n+1}-Q^n\rangle_F$ must be approximated.} Denote by $\boldsymbol{Q}_{11}^n$, $\boldsymbol{Q}_{12}^n$, and $\boldsymbol{r}^n$ the vectors of solutions at time $n\Delta t$, with $\Delta t$ the time step. 

The first term to be approximated is given by
\[
    f_1(\boldsymbol{Q}_{11}^n, \boldsymbol{Q}_{12}^n, \boldsymbol{r}^n)_i = (\boldsymbol{r}^n)_i P\left(\begin{bmatrix} (\boldsymbol{Q}_{11}^n)_i & (\boldsymbol{Q}_{12}^n)_i\\ (\boldsymbol{Q}_{12}^n)_i & -(\boldsymbol{Q}_{11}^n)_i\end{bmatrix}\right)\Delta x_i,\quad i=1,\dots,n,
\]
where $\Delta x_i$ is the volume of the finite element basis test function at node $i$. 

The second approximation is for
\begin{align*}
    f_2(\boldsymbol{Q}_{11}^n, \boldsymbol{Q}_{12}^n,\boldsymbol{Q}_{11}^{n+1}, \boldsymbol{Q}_{12}^{n+1})_i = &\\ &\hspace*{-3cm}\left\langle P\left(\begin{bmatrix} (\boldsymbol{Q}_{11}^n)_i & (\boldsymbol{Q}_{12}^n)_i\\ (\boldsymbol{Q}_{12}^n)_i & -(\boldsymbol{Q}_{11}^n)_i\end{bmatrix}\right), \begin{bmatrix}
        (\boldsymbol{Q}_{11}^{n+1})_i - (\boldsymbol{Q}_{11}^n)_i & (\boldsymbol{Q}_{12}^{n+1})_i - (\boldsymbol{Q}_{12}^n)_i\\
        (\boldsymbol{Q}_{12}^{n+1})_i - (\boldsymbol{Q}_{12}^n)_i & -(\boldsymbol{Q}_{11}^{n+1})_i + (\boldsymbol{Q}_{11}^n)_i
    \end{bmatrix}\right\rangle_F
  \end{align*}
for $i=1,\dots,n$. The weights in NEIM are chosen as $w_e(\mu_i;\mu_j) = \delta_{ij}$ and\reviewerA{, denoting $\mu_i = (a_i,t_i)$,}
\[
    w_t^{(k)}(a_i, t_i) = \begin{cases}
        1 &\text{if $a_i = a_k$ and \reviewerA{$|t_i - t_k| \le 5\Delta t$}},\\
        0 &\text{otherwise},
    \end{cases}
\]
We fit the reduced order models based on snapshots obtained via a finite element scheme with time step $\Delta t = 0.025$ to final time $T=2$. The snapshots correspond to 10 equally spaced values of $a$ between $-0.5$ and $0.5$, and the solutions at each time step for each of these values of $a$, giving 800 snapshots total. DEIM and NEIM are trained on all of these snapshots but are evaluated with just 4 modes each. 

We plot the behavior of the solutions in Figure \ref{fig:liquid_crystal_solutions}. The blue vectors in the figure denote the orientation of the liquid crystal molecules at a given point, also called the director. In particular, a solution $Q(t, x, y)$ represents a distribution of orientations of liquid crystal molecules at the point $(x,y)$, so there is an average of that distribution and a measure of how aligned molecules in that distribution are with the average. If $\lambda_1(t,x,y) \ge \lambda_2(t,x,y) = -\lambda_1(t,x,y)$ are the eigenvalues of $Q(t,x,y)$ and $v_1(t,x,y), v_2(t,x,y)$ are the corresponding eigenvectors, then the blue vectors are the eigenvector $v_1(t,x,y)$. Then the orange contour in the figure denotes how much light passes through the liquid crystal at a given point in space, which is determined by the average orientation of the liquid crystal molecules $v_1(t,x,y)$ and the measure of alignment of molecules in the distribution with the average, $\lambda_1(t,x,y)$. Specifically, the color is given by a color mapping with the scalar quantity
$
    \lvert \lambda_1(t,x,y)(v_1(t,x,y)\cdot (1,0))\rvert,
$
so if the molecules are in the vertical direction or are in the isotropic liquid phase (that is, they \reviewerCommon{are not} aligned with the average orientation so that $\lambda_1(t,x,y) \approx 0$) then light \reviewerCommon{does not} pass through the liquid crystal. We see that qualitatively, the high fidelity, DEIM, and NEIM solutions all have similar contour plots, and the blue director fields are also similar. The contour for DEIM appears more similar to the high fidelity contour than NEIM compared to the high fidelity contour for time $t=0.375$, but NEIM is more similar at the final time $t=2$. In both cases, the blue director fields in the DEIM and NEIM plots well-approximate the high fidelity director field, with each plot in the equilibrium state of all horizontal directors.

\begin{figure}[ht]
\begin{center}
\begin{tikzpicture}
    \draw (-5.55, 3) node {\textbf{High fidelity}};
    \draw (-0.8, 3) node {\textbf{DEIM}};
    \draw (3.85, 3) node {\textbf{NEIM}};
    \draw (0, 0) node[inner sep=0] {\includegraphics[scale=0.52]{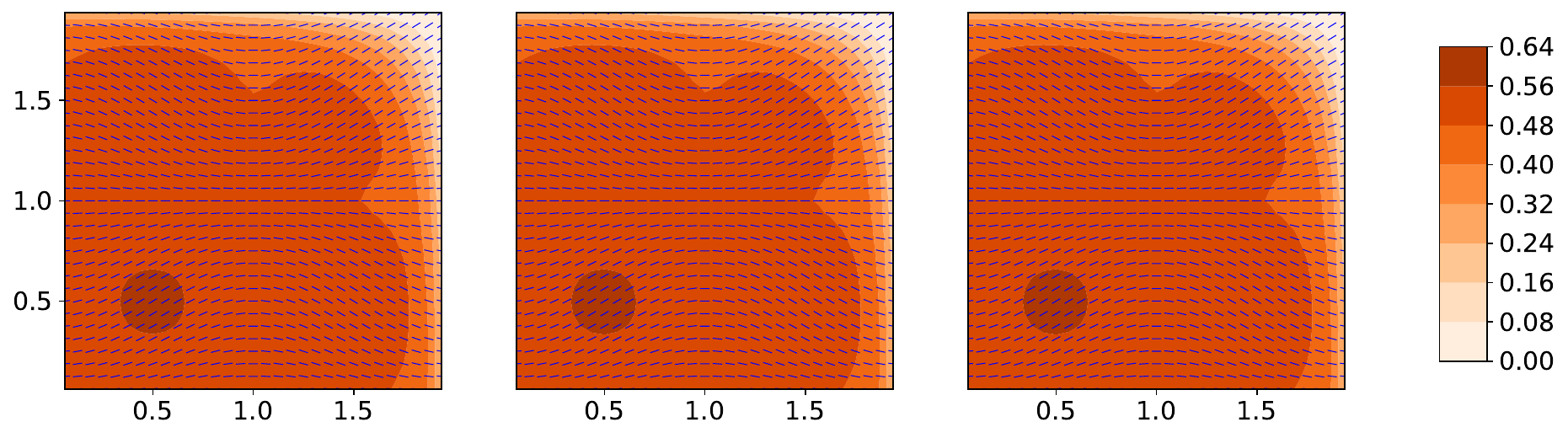}};
    \draw (0, -2.75) node {Solutions at $t=0$};
    \draw (0, -6) node[inner sep=0] {\includegraphics[scale=0.52]{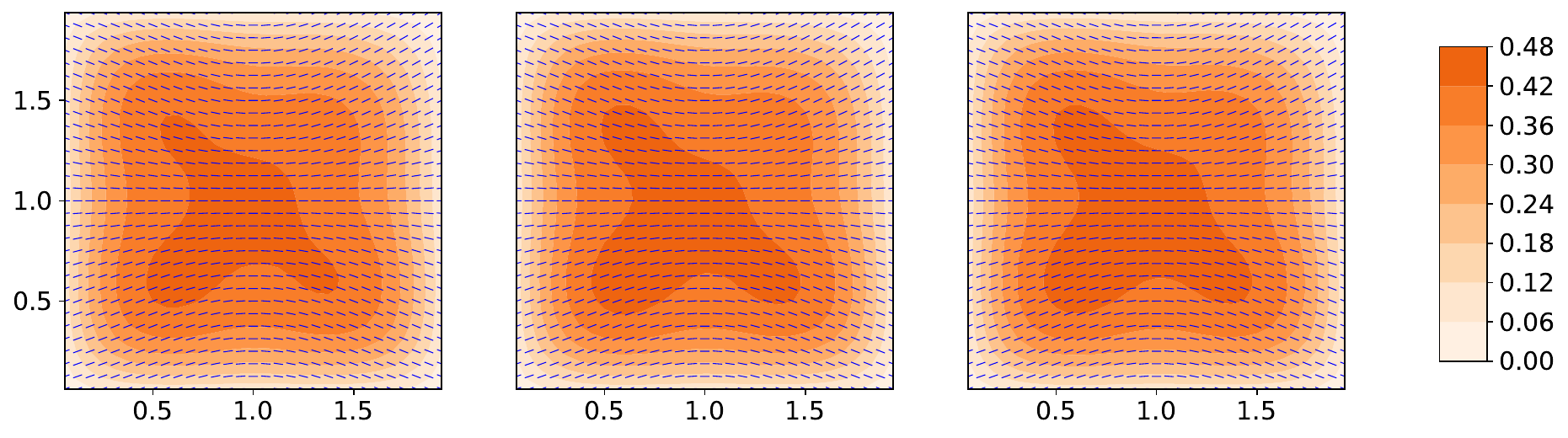}};
    \draw (0, -8.75) node {Solutions at $t=0.3$};
    \draw (0, -12) node[inner sep=0] {\includegraphics[scale=0.52]{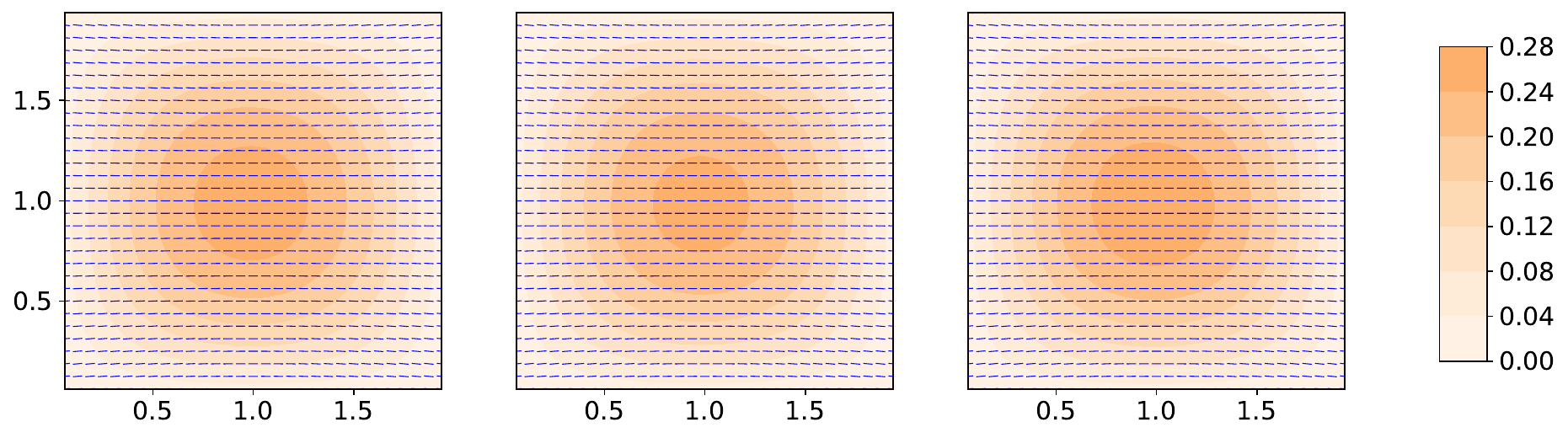}};
    \draw (0, -14.75) node {Solutions at $t=2$};
    \end{tikzpicture}
\end{center}
\caption{High fidelity, DEIM, and NEIM solutions for $a=-0.25$. The blue vectors give the eigenvector corresponding to the positive eigenvalue of the solution $Q$ (the orientation of the liquid crystal). The contour plot represents the intensity of light passing through the liquid crystal.}
\label{fig:liquid_crystal_solutions}
\end{figure}

In Figure \ref{fig:liquid_crystal_relative_errors_t}, we show the relative errors between the components of the DEIM and NEIM models using 4 modes and the high fidelity solution over time for $a=-0.25$, and we can observe that DEIM and NEIM give comparable results. The relative error for $r$ is much smaller than those of $Q_{11}$ and $Q_{12}$ because the scale of $r$ is much larger. Despite the relative error being around 10\% error in the worst case, mainly due to the complexity of the model and its time evolution, we already saw in Figure \ref{fig:liquid_crystal_solutions} that the solutions were all qualitatively the same, and we are still able to maintain similar accuracy for DEIM and NEIM in this relatively complex application.

In Figure \ref{fig:q11_relative_error_for_a}, we plot the relative error of $Q_{11}$, $Q_{12}$, and $r$ between the solutions computed by DEIM and NEIM and the high fidelity solution for different values of the parameter $a$. In all plots in the figure, DEIM and NEIM give errors which are comparable in magnitude. The error for NEIM is generally larger for $a>0$ than for $a<0$ in the plots of $Q_{11}$ and $Q_{12}$ relative error. This can be explained by the fact that the double-well potential is more complex for $a>0$. The DEIM relative error, on the other hand, decreases then increases as $a$ increases, with $a=0$ approximately the minimum error in each case. This is possibly due to the parameters close to the boundary of parameter space being more difficult to approximate. Now, in Figure \ref{fig:q11_error_contour_plot}, we plot the absolute and relative errors between the solutions $Q_{11}$ for DEIM and NEIM and the high fidelity solution for parameter pairs $(t,a)$. Qualitatively, we see that the error for NEIM is larger than that of DEIM, but the two still share a generally similar pattern in the plot of relative error, with positive $a$ and large time being worst-approximated, which reflects the sensitivity to the parameter samples, once again confirming the complexity of the model.

\begin{figure}[ht]
\begin{center}
  \includegraphics[width=.33\textwidth]{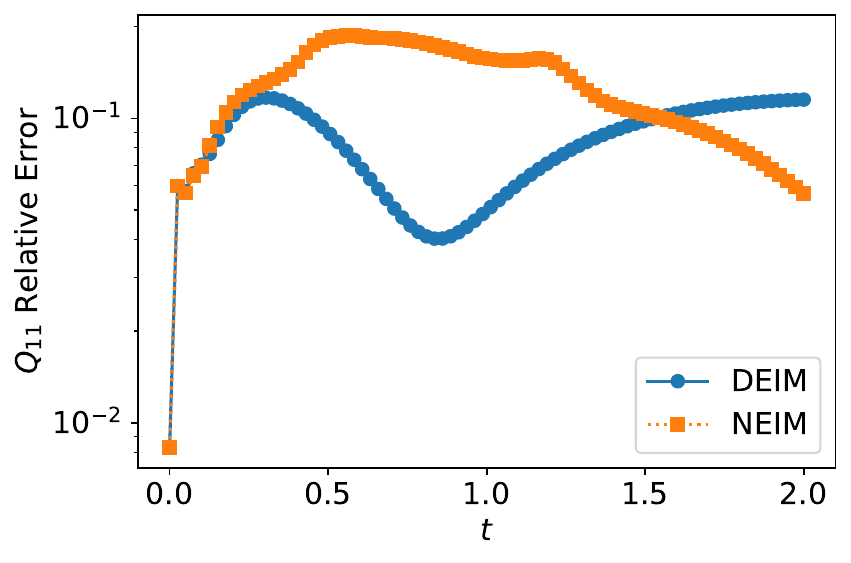}\hfill \includegraphics[width=.33\textwidth]{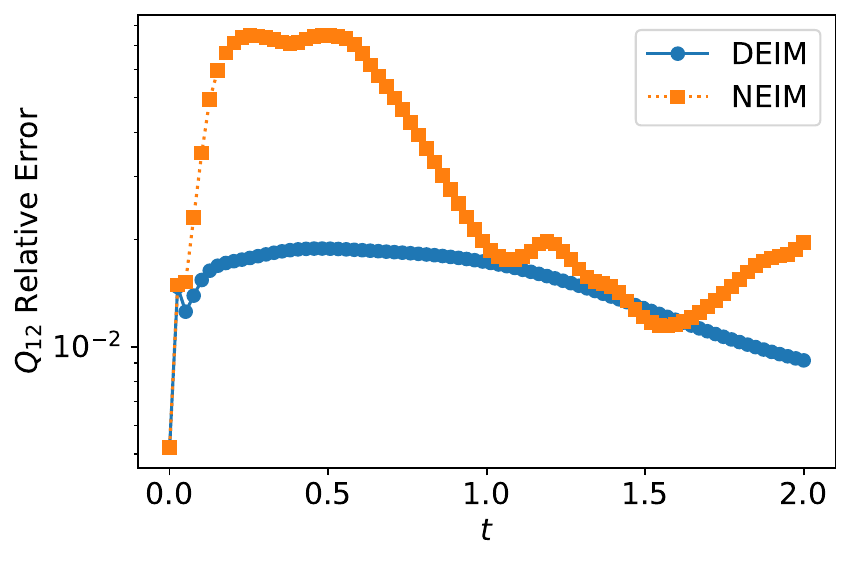}\hfill
\includegraphics[width=.33\textwidth]{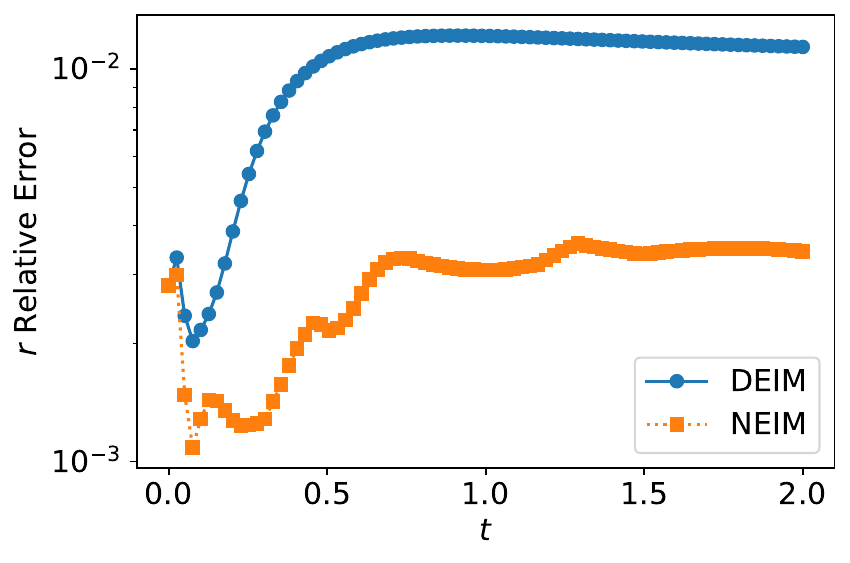}
\end{center}
\caption{Solution relative errors for $a=-0.25$.}
\label{fig:liquid_crystal_relative_errors_t}
\end{figure}

\begin{figure}[ht]
  \begin{center}
    \includegraphics[width=.33\textwidth]{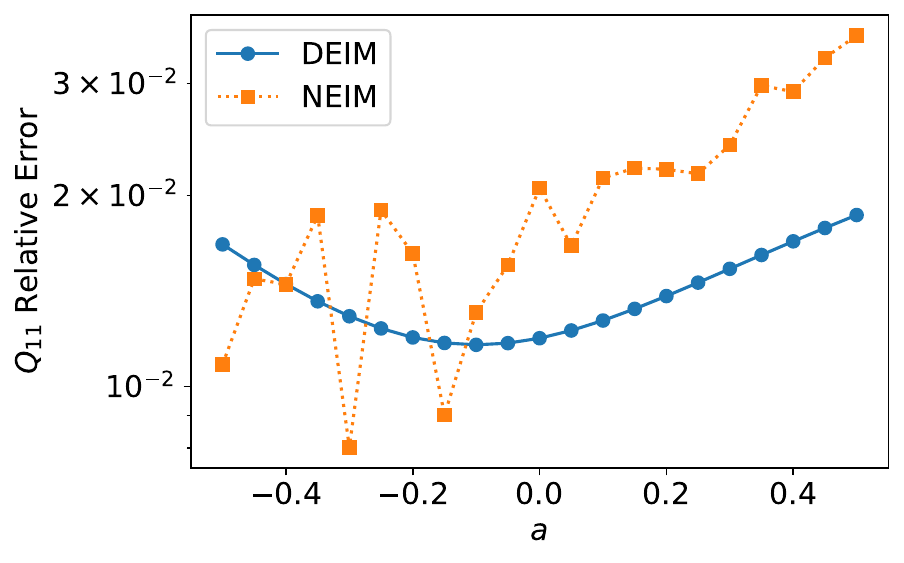}\hfill \includegraphics[width=.33\textwidth]{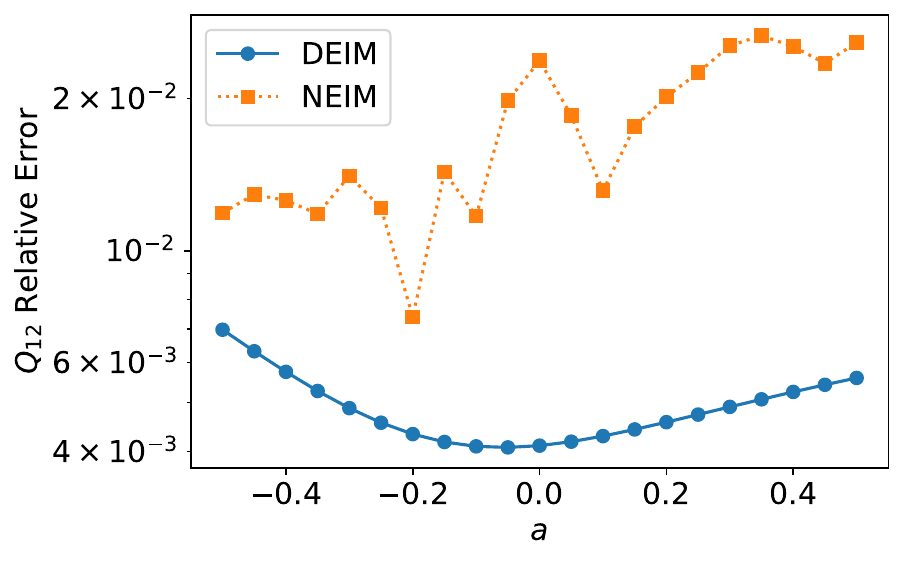}\hfill
 \includegraphics[width=.33\textwidth]{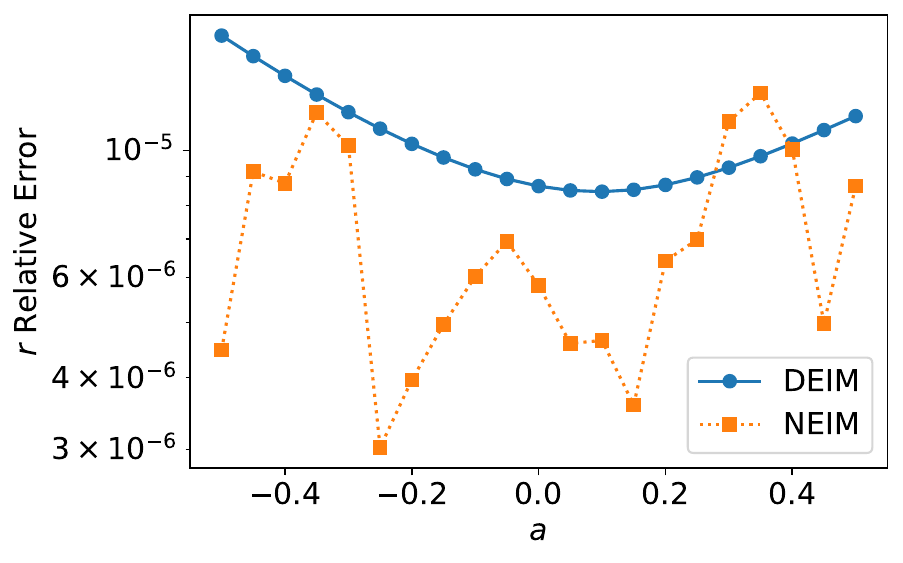}
\end{center}
    \caption{Relative errors for DEIM and NEIM by $a$.}
    \label{fig:q11_relative_error_for_a}
\end{figure}

\begin{figure}[ht]
    \centering
    \includegraphics[width=0.49\textwidth]{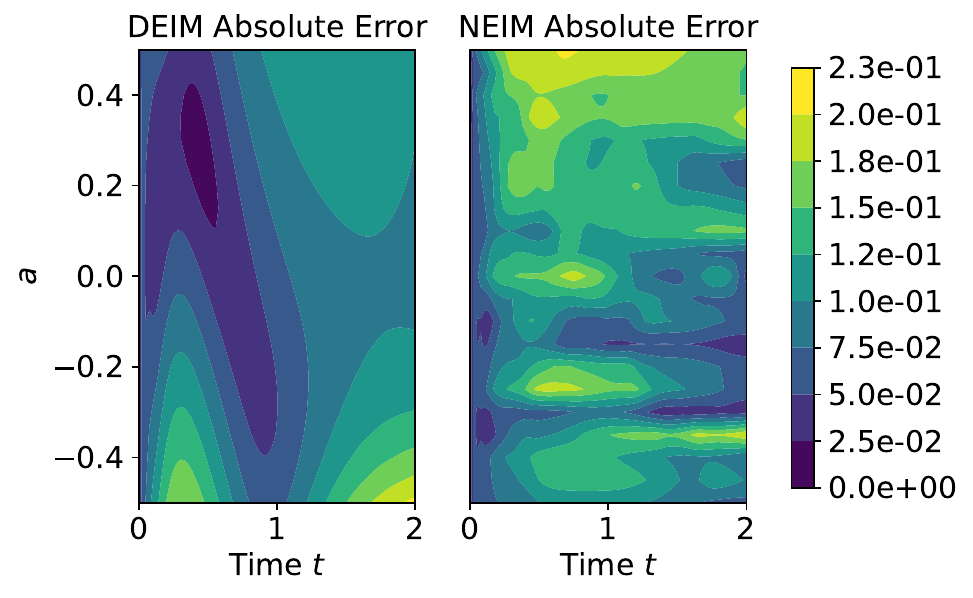}\hfill
    \includegraphics[width=0.49\textwidth]{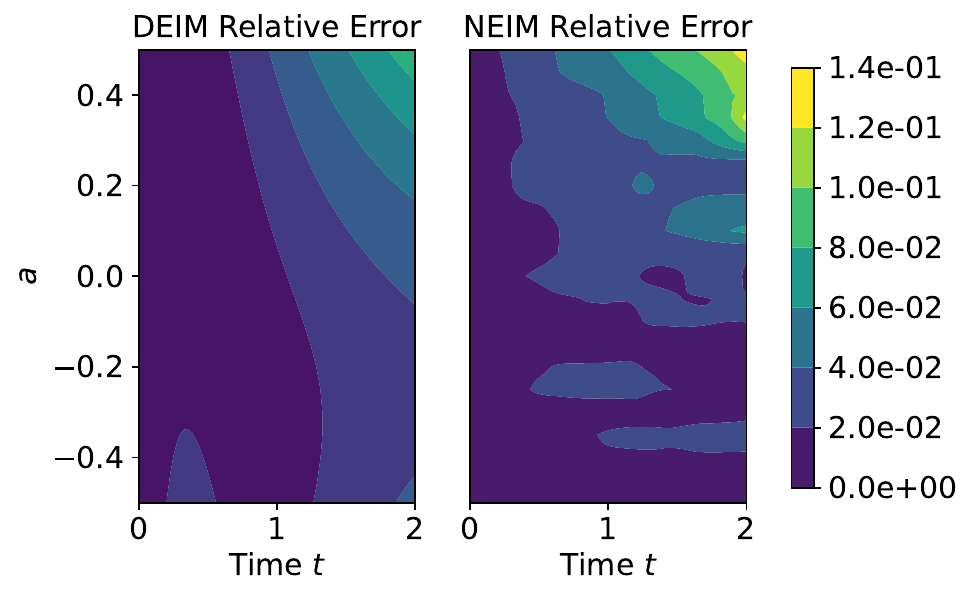}
    \caption{$Q_{11}$ absolute and relative errors for DEIM and NEIM for pairs $(t, a)$.}
    \label{fig:q11_error_contour_plot}
\end{figure}

\section{Conclusions}
\label{sec:conclusion}
The neural empirical interpolation method is a greedy method of using neural networks to approximate an affine decomposition of a nonlinearity in a reduced order model for a parameterized partial differential equation. It does this by iteratively adding to the NEIM approximation a coefficient times a neural network which fits an orthogonalized residual between the current approximation and the exact reduced nonlinearity. The coefficients are interpolated at the end of this iterative procedure. Thus, the algorithm can be thought of as a greedy method of training a DeepONet.

Because of the greedy nature of the algorithm, we were able to perform a basic error analysis of the method. We showed that the error on the training points used in the algorithm decreases as the number of terms in the approximate affine expansion increases. We also gave a decomposition of the error into a projection error, training error, and interpolation error to understand all the different sources.

We have discussed why NEIM could be preferable to the discrete empirical interpolation method  in several contexts. In particular, NEIM is a nonlinear projection, unlike DEIM which is linear. It is also efficient for nonlocal nonlinearities, unlike DEIM which works best for componentwise nonlinearities. Lastly, our method is easily compatible with automatic differentiation, but the discrete empirical interpolation method is not convenient for automatic differentiation when other software like a finite element package is needed to assemble the nonlinearity at each new evaluation.

We have shown via several numerical results that the methodology is effective when dealing with nonlinear terms coming from diverse applications. We showed it in two typical contexts arising from a finite difference discretizations of PDEs and in two more complex examples involving a physics-informed neural network for a nonlinear elliptic problem and a nonlinear parabolic PDE system. In each of these applications, NEIM was comparable to DEIM.

Lastly, because of its relation to DeepONets, NEIM lends itself to further extension and theoretical analysis. There is much further work to be done on NEIM with respect to optimally training the networks in the NEIM expansion and choosing weights in the NEIM algorithm, the latter being promising given the interpretation of the weights as being quadrature weights. Furthermore, the developed strategy could be used to explore further the interconnection between robust numerical algorithms required in the reduced order modelling context, and efficient data-driven methodologies, overcoming the difficulties related to intrusive strategies and extending the range of applicability towards more complex and real-world applications.  
\section*{Acknowledgements}
This material is based upon work supported by the National Science Foundation Graduate Research Fellowship Program under Grant No.\ DGE 2146752. Any opinions, findings, and conclusions or recommendations expressed in this material are those of the authors and do not necessarily reflect the views of the National Science Foundation. 
MH also acknowledges the ThinkSwiss Research Scholarship and the EPFL Excellence Research Scholarship. FP acknowledges the “GO for IT” program within the CRUI fund for the project “Reduced order method for nonlinear PDEs enhanced by machine learning”, and the support by Gruppo Nazionale di Calcolo Scientifico (INdAM-GNCS). This work has been conducted within the research activities of the consortium iNEST (Interconnected North-East Innovation Ecosystem), Piano Nazionale di Ripresa e Resilienza (PNRR) - Missione 4 Componente 2, Investimento 1.5 - D.D. 1058 23/06/2022, ECS00000043, supported by the European Union's NextGenerationEU program.
\bibliographystyle{plain}
\bibliography{NEIM.bib}

\end{document}